\documentclass[english,11pt]{article}
\usepackage{custom_tex}

\begin{document}
\title{Robust Inference Under Heteroskedasticity \\
via the Hadamard Estimator}

\date{\today}

\author{Edgar Dobriban, 
Weijie J.~Su,
Yachong Yang, and
Zhixiang Zhang\footnote{
Author affiliations: Department of Statistics and Data Science,  
University of Pennsylvania (ED, WJS, YY).
Department of Mathematics, 
University of Macau (ZZ).
E-mail addresses: 
\texttt{dobriban@wharton.upenn.edu},
\texttt{suw@wharton.upenn.edu},
\texttt{yachong@wharton.upenn.edu}
\texttt{zhixzhang@um.edu.mo}.
}}

\maketitle
\abstract{Drawing statistical inferences from large datasets in a  model-robust way is an important problem in statistics and data  science. In this paper, we 
propose methods that are robust to large and unequal noise in
different observational units (i.e., heteroskedasticity) for
statistical inference in linear regression. We leverage the
\emph{Hadamard estimator}, which is unbiased for the variances of
ordinary least-squares regression. This is in contrast to the popular
White's sandwich estimator, which can be substantially biased in high dimensions. 
We propose to estimate the signal strength, noise level, signal-to-noise
ratio, and mean squared error via the Hadamard estimator. We develop
a new degrees of freedom adjustment that gives more accurate
confidence intervals than variants of White's sandwich
estimator. Moreover, we provide conditions ensuring the estimator
is well-defined, by studying a new random matrix ensemble in which 
the entries of a random orthogonal projection matrix are squared. 
We also show approximate normality, using the second-order Poincar\'e inequality. Our work provides improved statistical theory and methods for linear regression in high dimensions. 
}

\tableofcontents

\section{Introduction}
Drawing statistical inferences from large datasets in a way that is robust to model assumptions is an important problem in statistics and data science. In this paper, we study a central question in this area, performing statistical inference for the unknown regression parameters in linear models. 

\subsection{Linear models and heteroskedastic noise} 

The linear regression model 
\begin{equation}\label{eq:linear_model}
Y = X\beta+\ep
\end{equation}
is widely used and fundamental in many areas. The goal is to understand the dependence of an outcome variable $Y$ on some $p$ covariates $x=(x_1,\ldots,x_p)^\top$. We observe $n$ such data points, arranging their outcomes into the $n\times 1$ vector $Y$, and their covariates into the $n\times p$ matrix $X$. We assume that $Y$ depends linearly on $X$, via some unknown $p\times 1$ parameter vector $\beta$. 
The noise vector $\ep$ consists of independent random variables.

A fundamental practical problem is that the structure of noise $\ep$ affects the accuracy of inferences about the regression coefficient $\beta$. If the noise level in an observation is very high, that observation contributes little useful information. Such an observation could bias our inferences, and we should discard or down-weight it. The practical meaning of large noise is that our model underfits the specific observation. However, we usually do not know the noise level of each observation. Therefore, we must design procedures that adapt to unknown noise levels, for instance by constructing preliminary estimators of the noise. This problem of unknown and unequal noise levels, i.e., \emph{heteroskedasticity}, has long been recognized as a central problem in many applied areas, especially in finance and econometrics. 

In applied data analysis, and especially in the fields mentioned
above, it is a common practice to use the ordinary least-squares (OLS) estimator $\hbeta = (X^\top X)^{-1} X^\top Y$ as the estimator of the unknown regression coefficients, despite the potential of heteroskedasticity. The OLS estimator is still unbiased, and has other desirable properties---such as consistency---under mild conditions. For statistical inference about $\beta$, the common practice is to use heteroskedasticity-robust confidence intervals. 

Specifically, in the classical low-dimensional case when the dimension $p$ is fixed and the sample size $n$ grows, the OLS estimator is asymptotically normal with asymptotic covariance matrix $C_\infty = \lim_{n\to\infty} n C $, with
\beq
\label{hbeta_cov}
C = \Cov{\hbeta} = (X^\top X)^{-1} X^\top \Sigma X(X^\top X)^{-1}.
\eeq
Here the covariance matrix of the noise is a diagonal matrix $\Cov{\ep} = \Sigma.$ To form confidence intervals for individual components of $\beta$, we need to estimate diagonal entries of $C$.  \cite{white1980heteroskedasticity}, in one of highest cited papers in econometrics, studied the following plug-in estimator of $C$, which simply estimates the unknown noise variances by the squared residuals: 
\beq
\label{w}
\hC_{\textnormal{W}} = (X^\top X)^{-1}X^\top \diag(\hep)^2 X (X^\top X)^{-1}.
\eeq
Here 
$\hep = Y-X\hbeta$
is the vector containing the residuals from the OLS fit. This is also
known as the \emph{sandwich estimator}, the \emph{Huber-White}, or the
\emph{Eicker-Huber-White} estimator. White showed that this estimator is consistent for the true covariance matrix of $\hbeta$, when the sample size grows to infinity, $n\to \infty$, with fixed 
dimension $p$. Earlier closely related work was done by \cite{eicker1967limit,huber1967behavior}. In theory, these works considered more general problems, but White's estimator was explicit and directly applicable to the central problem of inference in OLS. This may explain why White's work has achieved such a large practical impact, with more than 34,000 citations at the time of writing.

However, it was quickly realized that White's estimator is \emph{substantially biased} when the sample size $n$ is not too large---for instance when we only have twice as many samples as the dimension. This is a problem, because it can lead to incorrect statistical inferences. 
\cite{mackinnon1985some} proposed a bias-correction that is unbiased under homoskedasticity. However, the question of forming confidence intervals has remained challenging. Despite the unbiasedness of the MacKinnon-White estimate in special cases, confidence intervals based on it have below-nominal probability of covering the true parameters in low dimensions \citep[see e.g.,][]{kauermann2001note}.
It is not clear if this continues to hold in the high-dimensional case. 
In fact in our simulations we observe that these confidence intervals (CIs) can be anti-conservative in high dimensions. Thus, constructing accurate CIs in high dimensions remains a challenging open problem.

In this paper, we propose to construct confidence intervals via a variance
estimator that is unbiased \emph{even under heteroskedasticity}.
Since the estimator (described later), is based on Hadamard products,
we call it the \emph{Hadamard estimator}.  This remarkable estimator
has been discovered several times
\citep{hartley1969variance,chew1970covariance,cattaneo2017inference},
and the later works do not always appear to be aware of the earlier
ones. The estimator 
does not appear to be widely known by researchers in finance
and econometrics, and does not appear in standard econometrics
textbooks such as \cite{greene2003econometric}, or in recent review
papers such as \cite{imbens2016robust}. 
We came upon 
the
Hadamard estimator in 2017 while studying the bias of White's
estimator, and were surprised to find out about how early it was discovered. 
We emphasize that the papers above did not study many
of the important properties of this estimator. For instance, it is not even
clear based on these works under what conditions this estimator exists. 

In our paper, we start by showing how to solve five important problems in the linear regression model using the Hadamard estimator: constructing confidence intervals, estimating signal-to-noise ratio (SNR), signal strength, noise level, and mean squared error (MSE) in a robust way under heteroskedasticity (Section \ref{sec:five-import-probl}). To use the Hadamard estimator, we need to show the fundamental result that it is well-defined (Section \ref{sec:hadmard-estimator}). 
We prove matching upper and lower bounds on the relation between the dimension and sample size guaranteeing that the Hadamard estimator is generically well-defined. We also prove well conditioning. For this, we study a new random matrix ensemble in which the entries of a random partial orthogonal projection matrix are squared. 
Specifically, we prove sharp bounds on the smallest and largest eigenvalues of this matrix. This mathematical contribution should be of independent interest. 

Next, we develop a new degrees-of-freedom correction for the Hadamard estimator, which gives more accurate confidence intervals than several variants of the sandwich estimator  (Section \ref{sec:degr-freed-adjustm-1}).  Finally, we also establish the rate of convergence and approximate normality of the estimator, using the second-order Poincar\'e inequality (Section \ref{rate}).  We also perform numerical experiments to validate our theoretical results (Section \ref{sec:num}). Software implementing our method, and reproducing our results, is available from the authors' GitHub page, \url{http://github.com/dobriban/Hadamard}.

{\bf Notation.}
For a positive integer $n$, we denote $[n]=\{1,\ldots,n\}$.
For a vector $v \in \mathbb{R}^{n}$, let $\|v\|:=(\sum_{i=1}^n v_i^2)^{1/2}$ be the Euclidean norm. 
For any matrix $A \in \mathbb{R}^{m\times n}$, let $\|A\|$ or $\|A\|_{\op}$ stand for the operator norm, defined by $\| A\| := \sup_{v \in \mathbb{R}^{n}, v \neq 0} \|Av\|_2 / \| v\|_2$. 
The Frobenius norm is defined by $\|A\|_{\operatorname{\Fr}}:=(\sum_{i=1}^m \sum_{j=1}^n A_{ij}^2)^{1/2} $; and the infinity norm is $\|A\|_{\infty}:= \max_{1 \leq i \leq m} \sum_{j=1}^n |A_{ij}|$. 

\section{Main Results}
\label{me}

\subsection{Solving five problems under heteroskedasticity}
\label{sec:five-import-probl}

Under heteroskedasticity, some fundamental estimation and inference
tasks in the linear model are more challenging than
under homoskedasticity. As we will see, the difficulty often arises from a
lack of a good estimator of the variance of the OLS estimator. For the
moment, assume that there is an unbiased estimator of the
coordinate-wise variances of the OLS estimator. That
is, we consider a vector $\hV$ satisfying
$\E \hV = V$
under heteroskedasticity, where $V = \diag C = \diag\Cov{\hbeta}$ is
defined through equation
\eqref{hbeta_cov}. To define this unbiased estimator, we collect some useful
notation as follows, though the estimator itself shall be introduced in detail
in Section \ref{sec:hadmard-estimator}. Let $S = (X^\top X)^{-1} X^\top$  be the matrix used in defining the ordinary least-squares estimate, and $Q = I_n- X(X^\top X)^{-1}X^\top $ be the projection into the orthocomplement of the column space of $X$. Here $I_n$ is the identity matrix. Let us denote by $M\odot M$ the Hadamard---or elementwise---product of a matrix or vector $M$ with itself.

Among others, the following \textit{five} important applications
demonstrate the usefulness of the unbiased variance estimator $\hV$.

\vspace{-1em}
\paragraph{Constructing confidence intervals.} 
A first fundamental problem is inference for the regression coefficients. Assuming the noise $\varepsilon$ in the linear model
\eqref{eq:linear_model} follows a heteroskedastic normal distribution
$\ep\sim\N(0,\Sigma)$ for a diagonal covariance matrix $\Sigma$, the
random variable $(\hbeta_j - \beta_j)/\sqrt{V_j}$ follows the standard normal distribution. We replace the unknown variance $V_j$ of the OLS estimator
by its approximation $\hV_j$ and focus on the distribution of the
following \textit{approximate} pivotal quantity
\begin{equation}\label{eq:app_pivot}
\frac{\hbeta_j - \beta_j}{\sqrt{\hV_j}}.
\end{equation}
The distribution of this random
variable is approximated by a $t$ distribution in Section
\ref{sec:degr-freed-adjustm-1} and this plays a pivotal role in constructing confidence intervals and
conducting hypothesis testing for the coefficients. 
More generally, 
our inference result handles 
any linear combination---contrast---of $\beta$, see the result in Section \ref{rate}.

\vspace{-1em}
\paragraph{Estimating the SNR.}
Recall that $\|x\|= (\sum_i x_i^2)^{1/2}$ is the Euclidean norm of a vector $x\in \mathbb{R}^{n}$. 
The signal-to-noise ratio (SNR)
\[\snr = \frac{n \|\beta\|^2}{\E \|\ep\|^2}=\frac{n\|\beta\|^2}{\tr \Sigma}\]
of the linear model \eqref{eq:linear_model} is a fundamental measure
that quantifies the fraction of variability
explained by the covariates of an observational unit. In genetics, the SNR
corresponds to heritability if the response $y$ denotes the phenotype
of a genetic trait \citep{visscher2008heritability}. Existing work on
estimating this important ratio in linear models, however, largely
focuses on the relatively simple case of homoskedasticity (see, for
example, \cite{dicker2014variance,janson2017eigenprism}). Without
appropriately accounting for heteroskedasticity, the estimated SNR
may be unreliable. 

As an application of the estimator $\hV$, we propose to estimate the SNR using
\begin{equation}\label{eq:snr_estimator}
\widehat{\snr} = \frac{\|\hbeta\|^2 - 1_p^\top \hV}{n^{-1}1_p^\top (Q \odot Q)^{-1} (\hep \odot \hep)},
\end{equation}
where recall that $\hep$ is the vector of residuals in the linear
model, and $1_p$ denotes a column vector with all $p$ entries being
ones. Above, $(Q \odot Q)^{-1}$ denotes the inverse of the Hadamard
product $Q \odot Q$ of $Q=I_n- X(X^\top X)^{-1}X^\top$ with itself (we
will later study this invertibility in detail). The numerator and denominator of the fraction in
\eqref{eq:snr_estimator} are unbiased for the signal part and noise
part, respectively, as we show in the next two examples. 
As shown in Section \ref{sec:SNR},
this estimator is ratio-consistent.

\vspace{-1em}
\paragraph{Estimating signal squared magnitude.} A further fundamental
problem is estimating the magnitude of the regression coefficient
$\|\beta\|^2$. From the identity
$\E \|\hbeta\|^2 = \|\beta\|^2 + \tr\left(\Cov{\hbeta} \right)$,
it follows that an unbiased estimator of $\tr\left(\Cov{\hbeta} \right)$ is $1_p^\top
\hV$. Thus, an unbiased estimator of the squared signal
magnitude is $\|\hbeta\|^2 - 1_p^\top \hV$.

\vspace{-1em}
\paragraph{Estimating the total noise level.} As an intermediate step in the
derivation of the unbiased estimator $\hV$, we obtain the identity
\begin{equation}
\sigvec = (Q \odot Q)^{-1} \E(\hep \odot \hep).
\end{equation}
That is, the vector $\diag(\Sigma)$ of the entries of $\Sigma$ can be written as a matrix-vector product in the appropriate way.
As a consequence of this, we can use $(Q \odot Q)^{-1} (\hep \odot \hep)$
to estimate $\diag(\Sigma)$ in an unbiased way. In addition, we can use $1_p^\top (Q \odot Q)^{-1} (\hep \odot \hep)$
as an unbiased estimate of the total noise level $\tr(\Sigma) = \sum_{i=1}^n \mathrm{Var}(\ep_i)$.

\vspace{-1em}
\paragraph{Estimating the MSE.} An important problem concerning the least-squares method is estimating its mean squared error (MSE). Let 
$\mse = \E \|\hbeta - \beta\|^2$
be the MSE. Consider the estimator 
$\widehat{\mse}  = \sum_{i=1}^n \hV_i.$
As in the part ``Estimating signal squared magnitude,'' it follows that $\widehat{\mse}$ is an unbiased estimator of the MSE. Later in Section \ref{sec:num} we will show in simulations that this estimator is more accurate than the corresponding estimators based on White's and MacKinnon-White's covariance estimators. 

\subsection{The Hadamard estimator and its well-posedness}
\label{sec:hadmard-estimator}
This section specifies the variance estimator $\hV$. This estimator
has appeared in
\citet{hartley1969variance,chew1970covariance,cattaneo2017inference},
and takes the
following form of matrix-vector product
$$\hV = A (\hep\odot\hep),$$
where the matrix $A$ is 
\beq\label{ad}
A =  (S\odot S) (Q\odot Q)^{-1}.
\eeq
Here $(Q\odot Q)^{-1}$ is the usual
matrix inverse of  $Q\odot Q$ and recall that both $Q\odot Q$ and $\hep\odot\hep$ denote the
Hadamard product. As such, $\hV$ is henceforth referred to as the
\emph{Hadamard estimator}. In short, this is a method of moments estimator,
using linear combinations of the squared residuals. 

While the Hadamard estimator enjoys a simple expression, there is little work on a fundamental question: whether this estimator \textit{exists or not}. More precisely, in order for the
Hadamard estimator to be well-defined, the matrix $Q\odot Q$ must be
invertible. Without this knowledge, all five
important applications in Section \ref{sec:five-import-probl} would
suffer from a lack of theoretical foundation. While the invertibility
can be checked for a given dataset, knowing that it should hold 
under general conditions gives us a confidence that the method can 
work broadly.

As a major thrust of this paper, we provide a deep understanding of
under what conditions $Q\odot Q$ should be expected to be
invertible. The problem is theoretically nontrivial, because there are
no general statements about the invertibility of matrices whose entries are
squared values of some other matrix. In fact, $Q = I_n- X(X^\top
X)^{-1}X^\top $ is an $n\times n$ \emph{rank-deficient}
projection matrix of rank $n-p < n$. Therefore,
$Q$ itself is not invertible, and it is not clear how its rank behaves when the entries
are squared. However, we have the following lower bound on $n$ for this
invertibility to hold.
\begin{proposition}[Lower bound]\label{prop:known}
If the Hadamard product $Q \odot Q$ is invertible, then the sample size $n$ must be at least
\begin{equation}\label{eq:n_p_necessary}
n \ge p + \frac12 + \sqrt{2p + \frac14}.
\end{equation}
\end{proposition}

This result reveals that the Hadamard estimator simply \emph{does not exist}
if $n$ is only slightly greater than $p$, (say $p=n+1$), though the OLS estimator
exists in this regime. The proof of Proposition \ref{prop:known} comes
from a well-known property of the Hadamard product, that is, if a
matrix $B$ is of rank $r$, then the rank of $B \odot B$ is at most
$r(r+1)/2$ \citep[e.g.,][]{horn1991topics}. For completeness, a proof of this property is
given in Section \ref{pflb}. Using this property, the invertibility
of $Q \odot Q$ readily implies
\[
n \le \frac{(n - p)(n - p + 1)}{2},
\]
which is equivalent to \eqref{eq:n_p_necessary}.

In light of the above, it is tempting to ask whether
\eqref{eq:n_p_necessary} is sufficient for the existence of the
Hadamard estimator. In general, this is not the case. For example, let
$X =
\begin{pmatrix}
R\\
\bm 0
\end{pmatrix}$
for any orthogonal matrix $R \in \R^{p \times p}$. Then, $Q \odot Q$
is not invertible as $Q$ is a
diagonal matrix whose first $p$ diagonal entries are 0 and the
remaining are 1. This holds no matter how large $n$ is compared to
$p$.  However, such design matrices $X$ that lead to a degenerate $Q \odot Q$
are very ``rare'' in the sense of the following theorem. Recall that $Q = I_n - X(X^\top X)^{-1} X^\top$.
\begin{theorem}\label{thm:measure_0}
The set
\[
\left\{ X \in \R^{n \times p}: Q \odot Q \text{ does not have full rank} \right\}
\]
has Lebesgue measure zero in $\R^{np}$ if the inequality
\eqref{eq:n_p_necessary} is satisfied.
\end{theorem}

Therefore, the lower bound in Proposition \ref{prop:known} is sharp.
Roughly speaking, $n \ge p + c\sqrt{p}$, for $c>0$, is sufficient for the
invertibility of $Q \odot Q$. The proof of this result is new in the vast literature on the Hadamard
matrix product. In short, our proof uses certain algebraic
properties of the determinant of $Q \odot Q$ and employs a novel
induction step. Section \ref{sec:exist-hadam-estim} is devoted to developing the proof of
Theorem \ref{thm:measure_0} in detail. To be complete,
\cite{cattaneo2017inference} show high-probability invertibility when
$p>2n$ for Gaussian designs.
Our invertibility
result is more broadly applicable.

Up to now, we have conditioned on $X$, 
working in a fixed design setting. 
To better appreciate the theoretical contributions of our paper,
we consider a random matrix $X$ in the following corollary, which ensures that the Hadamard estimator is well-defined almost surely for popular random matrix ensembles of $X$ such as the Wishart ensemble.

\begin{corollary}\label{de}
Under the same conditions as in Theorem \ref{thm:measure_0}, if $X$ is
sampled from a distribution that is absolutely continuous with respect
to the Lebesgue measure on $\R^{n \times p}$ (put simply, $X$ has a
density), then $Q \odot Q$ is invertible almost surely.
\end{corollary}

Although $Q\odot Q$ is invertible under very general
conditions, our simulations reveal that the condition number of this
matrix can be very large for $p$ close to $n$ due to very small eigenvalues. This is problematic, because the estimator can then amplify the error.  Our next result shows that $Q\odot Q$ is well-conditioned under some conditions if $n > 2p$. We will show that this holds for certain random design matrices $X$. 

Suppose for instance that the entries of $X$ are i.i.d.~standard normal, $X_{ij}\sim \N(0,1)$. Then, each diagonal entry of $Q = I_n -  X(X^\top X)^{-1}X^\top $ is relatively large, of unit order. The off-diagonal entries are of order $1/n^{1/2}$. When we square the entries, the off-diagonal entries become of order $1/n$, while the diagonal ones are still of unit order. Thus, it is possible that the matrix is \emph{diagonally dominant}, so the diagonal entries are larger than the sum of the off-diagonal ones. This would ensure well-conditioning. We will show rigorously that this is true under some additional conditions.

Specifically, we will consider a \emph{high-dimensional asymptotic} setting, where the dimension $p$ and the sample size $n$ are both large. We assume that they grow proportionally to each other, $n,p \to \infty$ with $p/n\to\gamma>0$. This is a modern setting for high-dimensional statistics, and it has many connections to random matrix theory \citep[see e.g.,][]{bai2010spectral, paul2014random, yao2015large}.

We will provide bounds on the largest and smallest eigenvalues. We can handle correlated designs $X$, where each row is sampled i.i.d.~from a distribution with
$p\times p$
covariance matrix $\Gamma$. Let $\Gamma^{1/2}$ be the symmetric square root of $\Gamma$.

\begin{theorem}[Eigenvalue bounds for the Hadamard product with a random design] \label{thm:Tnorm}
Suppose the rows $x_i$ of $X$ are i.i.d.~and have the form $x_i =
\Gamma^{1/2} z_i$, where $z_i$ have i.i.d.~entries with mean zero,
unit variance and uniformly bounded $(8+\delta)$-th moment. Suppose
that $\Gamma$ is invertible. 
Then, as $n,p\to\infty$ such that
for $ \gamma_{p,n}:=p/n$, we have 
$\limsup \gamma_{p,n} < 1/2$,
the matrix $T=Q\odot Q$ 
with  $Q = I_n- X(X^\top X)^{-1}X^\top $
satisfies the following eigenvalue bounds  
for any fixed $\xi>0$ and sufficiently large $n$:
\begin{equation*}
\lambda_{\max}(T)<1-\gamma_{p,n}+\xi,\end{equation*}
and 
\begin{equation*}
    \lambda_{\min}(T) > (1-\gamma_{p,n})(1-2 \gamma_{p,n})-\xi,
\end{equation*}
with probability at least $1-Cn^{-1-\delta/4} \xi^{-4-\delta/2}$ for some positive constant $C$ not depending on $\xi$.
\end{theorem}

See Section \ref{pf:dd} for a proof. 
Hence,
if $p/n\to \gamma <1/2$, then almost surely
\[
(1-\gamma)(1-2\gamma) \le \liminf\lambda_{\min}(T) \le
\limsup\lambda_{\max}(T) \le (1-\gamma).
\]
Practically speaking, 
the above result states that
the condition number of $T$ is at most
$1/(1-2\gamma)$ with high probability. 
Our invertibility results are stronger than those of  \cite{cattaneo2017inference}. Specifically, we show generic invertibility in finite dimensional designs with probability one, and condition number bounds on non-Gaussian correlated designs that go beyond those considered in their work.

While Corollary \ref{de} proves invertibility for 
continuous distributions, in practice 
some columns of $X$ can be discrete.\footnote{We thank a referee for raising this point.}
In that case, we can still obtain invertibility or condition number bounds by 
applying Lemma \ref{qf} used in the proof of Theorem \ref{thm:Tnorm}. 
That result which has no assumptions on the continuity of $X$ and provides non-asymptotic bounds for the eigenvalues of the matrix. 

As an illustration, we study an ANOVA design.
For two integers $a \le  b$,
we write $[a : b] = \{a, a + 1 \cdots, b\}.$ 
Consider a sequence $1\equiv n_0< n_1 < n_2 < \cdots < n_p \equiv n$.
Consider the ANOVA design where  
$X_{j,1} =1$ for $j \in [1,n_1]$, 
$X_{j,2} =1$ for $j \in [n_1+1, n_2]$,
and so on, 
until $X_{j,p} = 1$ for $j \in [n_{p-1}+1,  n_p]$.
It is readily verified that $x_i^\top R_i^{-1} x_i = 1/(n_j-n_{j-1}-1)$ if $i\in[n_{j-1},n_j]$.
If  $3\le \min_{j\in[p]} (n_{j}-n_{j-1})\le \max_{j\in[p]} (n_{j}-n_{j-1})\le  C$, then $\lambda_{\min}(T)\ge 2/9$ and $\lambda_{\max}(T) \le C/(C+1)$.
Hence, we obtain a bound on
the condition number of $T$.

\subsection{Degrees-of-freedom adjustment}
\label{sec:degr-freed-adjustm-1}

To obtain a confidence interval for $\beta_j$, we propose to approximate the
distribution of the approximate pivot in \eqref{eq:app_pivot} by a $t$-distribution. The key is to find a good approximation to the degrees of freedom. Let us denote by $V_j = \Var{\widehat\beta_j}$, the expected value of $\hV_j$. Suppose the degrees of freedom of $\hV_j$ are $d_j$. Using the second moment properties of the $\chi^2_{d_j}$ variable, these degrees of freedom should obey that
\[
\E \hV_j^2 \approx \frac{V_j^2}{d_j^2} \E \chi^4_{d_j} = V_j^2 (1 + 2/d_j). 
\]
Consequently, we formally define
\begin{equation}\label{eq:dof}
d_j = \frac2{\frac{\E \hV_j^2}{V_j^2} - 1} = \frac{2V_j^2}{\E \hV_j^2 - V_j^2}.
\end{equation}
To proceed, we need to evaluate $\E [\hV \odot \hV] \in \R^p$. The
following proposition gives a closed-form expression of this vector assuming \textit{homoskedasticity}. Let us denote
\beq\label{E}
E =  \diag\left[(X^\top X)^{-1} \right] \odot \diag\left[(X^\top X)^{-1} \right].
\eeq
Recall that $S = (X^\top X)^{-1} X^\top$.

\begin{proposition}[Degrees of freedom]
If the noise $\ep$ has i.i.d. normal entries, we have that the vector of degrees of freedom of $\hV$, defined in equation \eqref{eq:dof}, has the form
\begin{equation}\label{eq:d_formula}
d = \frac{2E}
{\diag\left[ (S\odot S ) 1_n 1_n^\top (S\odot S)^\top \right]
+ 2 \diag\left[(S \odot S) (Q \odot Q)^{-1} (S \odot S)^\top \right] 
- E},
\end{equation}
where the division is understood to be entrywise.
\label{df}
\end{proposition}

See Section \ref{dfpdf} for a proof. 

We call the inference method based on approximating 
$(\hat{\beta}_i- \beta_i)/\hat{V}_i^{1/2}$
by
a $t$-distribution with the degrees of freedom specified by \eqref{eq:d_formula} the \emph{Hadamard-t} method.
This result also leads to a useful \emph{degrees of freedom} heuristic. If the degrees of freedom $d_i$ are large, this suggests that inferences for $\beta_i$ are based on a large amount of information. On the other hand, if the degrees of freedom are small, this suggests that the inferences are based on little information, and may thus be unstable. 

In our case, the $t$-distribution is still a heuristic, because the numerator and denominator are not independent under heteroskedasticity. However, the degree of dependence can be bounded as follows:
\begin{eqnarray}
\|\Cov{\hbeta,\hep}\|_{\op} &=  \|S\Sigma (S^\top X^\top - I)\|_{\op}  = \|S(\Sigma - cI) (S^\top X^\top - I)\|_{\op} \nonumber\\
 &\le   \|S\|_{\op}\| \Sigma - cI \|_{\op} \| S^\top X^\top - I\|_{\op} \le \frac{|\Sigma_{\max} - \Sigma_{\min}|}{2\sigma_{\min}(X)}.\label{covbd}
\end{eqnarray}
In the first line, we have used that $S = (X^\top X)^{-1} X^\top$. Hence, $S( S^\top X^\top - I) = 0$. Indeed, 
$$
SS^\top X^\top = (X^\top X)^{-1} X^\top X (X^\top X)^{-1} X^\top = (X^\top X)^{-1} X^\top = S.
$$
For this reason, we can add a constant times $S( S^\top X^\top - I) = 0$ in the second step. Then, we can use the inequality $\|AB\|_{\op} \le \|A\|_{\op} \|B\|_{\op}$ for any two conformable matrices $A,B$.

In \eqref{covbd}, we have chosen $c = (\Sigma_{\max} + \Sigma_{\min})/2$, where $\Sigma_{\max}$ and $\Sigma_{\min}$ denote the maximal and minimal entries of $\Sigma$, respectively. Moreover, we have also used that $\|S\|_{\op} = 1/\sigma_{\min}(X)$, while $\| S^\top X^\top - I\|_{\op} = \| X (X^\top X)^{-1} X^\top - I\|_{\op} \le 1$.

Now, for designs $X$ of aspect ratios $n \times p$ that are not close to 1, and with i.i.d.~entries with sufficiently many moments, it is known that $\sigma_{\min}(X)$ is of the order $n^{1/2}$. This suggests that the covariance between $\hbeta$ and $\hep$ is small. Hence, this heuristic suggests that the $t$-approximation should be accurate. Moreover when $\hV_j - V_j  \to 0$ in probability, and under the conditions in Section \ref{sec:norm}, we also have that the limiting distribution is standard normal.

\subsection{Hadamard estimator with 
\texorpdfstring{$p = 1$}{p = 1}}
\label{sec:hadam-estim-with}
As a simple example, consider the case of one covariate, when $p=1$. In this case, we have $Y = X\beta + \ep$, where $y,X,\ep$ are $n$-vectors. Assuming without loss of generality that $X^\top X = 1$, the OLS estimator takes the form $\hbeta = X^\top y$. Its variance equals $V = \sum_{j=1}^n X_j^2 \Sigma_{j}$, where $\Sigma_j$ is the variance of $\ep_j$, and $X_j$ are the entries of $X$. 

The Hadamard estimator takes the form 
$$\hV = \frac{ \sum_{j=1}^n  \frac{X_j^2}{1-2X_j^2} \hep_j^2}{1+\sum_{j=1}^n  \frac{ X_j^4}{1-2X_j^2}},$$
which is well-defined if all coordinates $X_j^2$ are small enough that $1-2X_j^2>0$. See section \ref{p1pf} for the argument. The unbiased estimator is not always nonnegative. To ensure nonnegativity, we need $X_j^2 < 1/2$ in this case. In practice, we may enforce non-negativity by using $\max(\hV,0)$ instead of $\hV$, but see below for a more thorough discussion.

For comparison, White's variance estimator is
$\hV_{\mathrm{W}} = \sum_{j=1}^n X_j^2  \hep_j^2,$
while MacKinnon-White's variance estimator \citep{mackinnon1985some} can be seen to take the form
$$\hV_{\mathrm{MW}} = \sum_{j=1}^n \frac{X_j^2}{1-X_j^2}  \hep_j^2 = \sum_{j=1}^n \frac{X_j^2}{ \sum_{i=1, i\neq j}^nX_i^2}  \hep_j^2.$$

We observe that each variance estimator is a weighted linear combination of the squared residuals, where the weights are some functions of the squares of the entries of the feature vector $X$. For White's estimator, the weights are simply the squared entries. For MacKinnon-White's variance estimator, the weights are scaled up by a factor $1/(1-X_j^2)>1$. As we know, this ensures the estimator is unbiased under homoskedasticity. For the Hadamard estimator, the weights are scaled up more aggressively by $1/(1-2X_j^2)>1$, and there is an additional normalization step. In general, these weights do not have to be larger---or smaller---than those of the other two weighting schemes. 

A critical issue is that \emph{the Hadamard estimator may not always be non-negative}. It is well known that unbiased estimators may fall outside of the parameter space \citep{lehmann1998theory}. When $p=1$, almost sure non-negativity is ensured when the coordinates of $X$ are sufficiently small. It would be desirable, but seems non-obvious, to obtain such results for general dimension $p$. 

In addition, the degrees of freedom from \eqref{eq:d_formula} simplifies to
$d = 1+1/\left(\sum_{j=1}^n  \frac{ X_j^4}{1-2X_j^2}\right).
$
This can be as large as $n-1$, for instance $d=n-1$ when all $X_i^2=1/n$. The degrees of freedom can only be small if the distribution of $X_i^2$ is very skewed.

\subsection{Bias of classical estimators}
\label{sec:bias-class-estim}
As a byproduct of our analysis, we also obtain explicit formulas for the bias of the two classical estimators of the variances of the ordinary least-squares estimator, namely the White and MacKinnon-White estimators. This can in principle enable us to understand when the bias is small or large.

The estimator proposed by \cite{mackinnon1985some}, which we will call the \emph{MW estimator}, is:
\beq
\label{mw}
\hC_{\mathrm{MW}} =(X^\top X)^{-1}[X^\top \hSigma_{\mathrm{MW}} X] (X^\top X)^{-1},
\eeq
where $\hSigma_{\mathrm{MW}} = \diag(Q)^{-1} \diag(\hep)^2$. This estimator is unbiased under homoskedasticity, that is, $\Sigma = \sigma^2 I_n$. It is denoted as HC2 in the paper \cite{mackinnon1985some}. The same estimator was also proposed by \cite{wu1986jackknife}, equation (2.6).

\begin{proposition}[Bias of classical estimators]
Consider White's covariance estimator defined in \eqref{w} and MacKinnon-White's estimator defined in \eqref{mw}. Their bias for estimating the coordinate-wise variances of the OLS estimator equals, respectively
\beq\label{bw}
b_{\mathrm{W}} = (S\odot S)  [(Q \odot Q)-I_n] \vec\Sigma
\eeq
for White's covariance estimator, and
\beq\label{bmw}
b_{\mathrm{MW}} = (S\odot S)  [\diag(Q)^{-1}(Q \odot Q)-I_n] \vec\Sigma
\eeq
for MacKinnon-White's estimator. Here $\vec\Sigma$ is the vector of diagonal entries of $\Sigma$, the covariance of the noise.
\label{bc}
\end{proposition}

See Section \ref{pf:bc} for a proof.
In particular, MacKinnon-White's estimator is known to be unbiased
under homoskedasticity, that is when $\Sigma = I_n$ \citep{mackinnon1985some}. This can be checked easily using our explicit formula for the bias. Specifically suppose that $\Sigma = I_n$. Then, $\vec\Sigma = 1_n$, the vector of all ones. Therefore, $(Q \odot Q) \vec\Sigma  = \vect(\|q_j\|^2)$, the vector of squared Euclidean norms of the rows of $Q$. Since $Q$ is a projection matrix, $Q^2=Q$, so $\|q_j\|^2 = Q_{jj}$. Therefore we see that 
$$ [\diag(Q)^{-1}(Q \odot Q)-I_n] \vec\Sigma = \diag(Q)^{-1} \vect(Q_{jj}) - 1_n = 0,$$
so that MacKinnon-White's estimator is unbiased under  homoskedasticity.

\subsection{Some related work}  
\label{relw}

There has been a lot of related work on inference in linear models under heteroskedasticity. Here we can only mention a few of the most closely related works, and refer to \cite{imbens2016robust} for a review. In the low-dimensional case, \cite{bera2002some} compared the Hadamard and White-type estimators and discovered that the Hadamard estimator lead to more accurate coverage, while the White estimators have better mean squared error.

As a heuristic to improve the performance of the MacKinnon-White (MW) confidence intervals in high dimensions, \cite{bell2002bias} have a similar approach to ours, with a $t$ degrees of freedom correction. Simulations in the very recent review paper by \cite{imbens2016robust} 
suggest this method is the state of the art for heteroskedasticity-consistent inference, and performs well under many settings.  However, this correction is computationally more burdensome than the MW method, because it requires a separate $O(p^3)$ computation for each regression coefficient, raising the cost to $O(p^4)$. In contrast, our method has computational cost $O(p^3)$ only. In addition, the accuracy of their method  typically does not increase substantially compared to the MW method. We think that this could be due to the bias of the MW method under heteroskedasticity.

In this work, we have used the term ``robust'' informally to mean insensitivity to assumptions about the covariance of the noise. Robust statistics is a much larger field which classically studies robustness to outliers in the data distribution \citep[e.g.,][]{huber2011robust}. Recent work has focused, among many other topics, on high-dimensional regression and covariance estimation 
 \cite[e.g.,][etc]{el2013robust,chen2016general,donoho2016high,zhou2018new,diakonikolas2017being}.

\section{Existence of Hadamard estimator}
\label{sec:exist-hadam-estim}

In this section we develop the novel proof of the existence of the
Hadamard estimator. We begin by observing that Theorem \ref{thm:measure_0} is equivalent to the proposition below. This is because the Lebesgue measure admits an orthogonal decomposition using the SVD.
\begin{proposition}\label{prop:hadamard_full}
Assume $r(r+1)/2 \ge n$. Denote by $\mathcal{Q}$ the set of all $n \times n$ projection matrices of rank $r$ and let $\d \mathcal{Q}$ be the Lebesgue measure on $\mathcal{Q}$. Then, the set
$\{Q \in \mathcal{Q}: \rank(Q \odot Q) < n\}$
has zero-$\d\mathcal{Q}$ measure.
\end{proposition}

We take the following lemma as given for the moment.
\begin{lemma}\label{lm:exists}
Under the same assumptions as Proposition \ref{prop:hadamard_full}, there exists a $Q^\ast \in \mathcal{Q}$ such that
$\rank(Q^\ast \odot Q^\ast) = n$.
\end{lemma}

A proof of Proposition \ref{prop:hadamard_full} using Lemma \ref{lm:exists} is readily given as follows.
\begin{proof}[Proof of Proposition \ref{prop:hadamard_full}]
Let $p = n - r$. Consider the map from $\R^{n \times p}$ (ignoring the zero-Lebesgue measure set where $X$ is not of rank $p$) to $\mathcal Q$:
\[
X \in \R^{n \times p} \longrightarrow Q = I - X(X^\top X)^{-1} X^\top \in \mathcal Q.
\]
It is easy to see that the map is a surjection and the preimage of this map for every $Q \in \mathcal{Q}$ is rotationally equivalent to each other. Hence, it suffices to show that the set of $X$ where the Hadamard product of $I - X(X^\top X)^{-1} X^\top$ is degenerate is measure zero.

We observe that the determinant takes the form
\[
\det\left( (I - X(X^\top X)^{-1} X^\top) \odot (I - X(X^\top X)^{-1} X^\top) \right) = \frac{f_1(X)}{f_2(X)},
\]
where $f_1(X)$ and $f_2(X)$ are polynomials in  the $np$ variables $X_{ij}, 1\le i \le n, 1\le j \le p$. As a fundamental property of polynomials, one and exactly one of the following two cases holds:
\begin{enumerate}
\item[(a)] The polynomial $f_1(X) \equiv 0$ for all $X$.
\item[(b)] The roots of $f_1(X)$ are of zero Lebesgue measure.
\end{enumerate}

Lemma \ref{lm:exists} falsifies case (a). Therefore, case (b) must hold. Recognizing that the set of $X$ where the Hadamard product of $Q(X)$ is not full rank is a subset of the roots of $f_1(X)$, case (b) confirms the claim of the present lemma.


\end{proof}

Now we turn to prove Lemma \ref{lm:exists}. For convenience, we adopt the following definition.
\begin{definition}\label{def:o_rank}
For a set of vectors $u_1, \ldots, u_r \in \R^n$, write $\rank^{\odot}(u_1, \ldots, u_r)$ the rank of the $r(r+1)/2$ vectors each taking the form $u_i \odot u_j$ for $1 \le i \le j \le r$.
\end{definition}
First, we give two simple lemmas.
\begin{lemma}\label{lm:equi}
Suppose two sets of vectors $\{u_1, u_2, \ldots, u_{r}\}$ and $\{u'_1, u'_2, \ldots, u'_{r'}\}$ are linearly equivalent, meaning that one can be linearly represented by the other. Then,
\[
\rank^{\odot}(u_1, \ldots, u_r) = \rank^{\odot}(u'_1, \ldots, u'_{r'}).
\]
\end{lemma}

\begin{lemma}\label{lm:rank_identity}
For any matrix $P$ that takes the form $P = u_1 u_1^\top + \ldots + u_r u_r^\top$ for some vectors $u_1, \ldots, u_r$, we have
\[
\rank(P \odot P) = \rank^{\odot} (u_1, \ldots, u_r).
\]
\end{lemma}

Making use of the two lemmas above, Lemma \ref{lm:exists} is validated once we show the following.
\begin{lemma}\label{lm:unrestricted}
There exist vectors $u_1, \ldots, u_r$ such that
$\rank^{\odot}(u_1, \ldots, u_r) = n$
if $r(r+1)/2 \ge n$. 
\end{lemma}
To see this point, we apply the Gram--Schmidt orthonormalization to $u_1, \ldots, u_r$ considered in Lemma \ref{lm:unrestricted}, and get orthonormal vectors $v_1, \ldots, v_r$. Write $Q^\ast = v_1 v_1^\top + \ldots + v_r v_r^\top$, which belongs to $\mathcal{Q}$. Since $u_1, \ldots, u_r$ and $v_1, \ldots, v_r$ are linearly equivalent, Lemmas \ref{lm:equi} and \ref{lm:rank_identity} reveal that
\[
\rank(Q^\ast \odot Q^\ast) = \rank^{\odot}(v_1, \ldots, v_r) = \rank^{\odot}(u_1, \ldots, u_r) = n.
\]

Now we aim to prove Lemma \ref{lm:unrestricted}.
\begin{proof}[Proof of Lemma \ref{lm:unrestricted}]

We consider a stronger form of Lemma \ref{lm:unrestricted}: for
\textit{generic} $u_1, \ldots, u_r$, any combination of $n$ vectors
from $u_i \odot u_j$ for $1 \le i \le j \le r$ have full rank. Here
\textit{generic} means that this statement does not hold only for a set of zero Lebesgue measure.

We induct on $n$. The statement is true for $n = 1$. Suppose it has been proven true for $n - 1$. Let $\mathcal{U}$ denote an arbitrary subset of $\{(i, j): 1 \le i \le j \le r\}$ with cardinality $n$. Write
$P = (u_i \odot u_j)_{(i,j) \in \mathcal U}$.

It is sufficient to show that $\det(P)$ is generically nonzero. As earlier in the proof of Proposition \ref{prop:hadamard_full}, it suffices to show that $\det(P)$ is not \textit{always} zero. Without loss of generality, let $(i_0, j_0) \in \mathcal{U}$ be the first column of $P$. Expressing the determinant of $P$ in terms of its minors along the first column, we see that $\det(P)$ is an affine function of $u_{i_0}(1) u_{j_0}(1)$, with the leading coefficient being the determinant of a $(n - 1) \times (n - 1)$ minor matrix that results from $P$ by removing the first row and the first column. The induction step is complete if we show that this minor matrix, denoted by $P_{1,1}$ is nonzero generically. Write $u_i^{(-1)}$ the vector in $\R^{n-1}$ formed by removing the first entry from $u_i$ for $i = 1, \ldots, r$. Then, each of the $n-1$ column of $P_{1,1}$ takes the form $u_i^{(-1)} \odot u_j^{(-1)}$ for some $(i, j) \in \mathcal U \setminus \{(i_0, j_0)\}$. Since the induction step has been validated for $n-1$, it follows that the determinant of $P_{1,1}$ is nonzero in the generic sense.

\end{proof}

To complete this section, we prove below Lemmas \ref{lm:equi} and \ref{lm:rank_identity}.
\begin{proof}[Proof of Lemma \ref{lm:equi}]
Since $\{u'_1, u'_2, \ldots, u'_{r'}\}$ can be linearly represented by $\{u_1, u_2, \ldots, u_{r}\}$, each $u'_j$ can be written as
$u'_j = \sum_{l=1}^r a_l^j u_l$
for constants $a_l^j$. Using the representation, the Hadamard product between two vectors reads
\[
\begin{aligned}
u'_i \odot u'_j &= \left( \sum_{l=1}^r a_l^i u_l \right) \odot \left( \sum_{l=1}^r a_l^j u_l \right)
                      = \sum_{l_1, l_2} a_{l_1}^i a_{l_2}^j u_{l_1} \odot u_{l_2}.
\end{aligned}
\]
This expression for $u'_i \odot u'_j$ suggests that $u'_i \odot u'_j$ is in the linear span of $u_{l _1} \odot u_{l_2}$ for $1 \le l_1 \le l_2 \le r$. As a consequence of this, it must hold that
\[
\begin{aligned}
\rank^{\odot}(u'_1, u'_2, \ldots, u'_{r'}) &\equiv \rank(\{u'_i \odot u'_j: 1 \le i \le j \le r'\})\\
& \le \rank(\{u_{l_1} \odot u_{l_2}: 1 \le l_1 \le l_2 \le r\})
 = \rank^{\odot}(u_1, u_2, \ldots, u_{r}).
\end{aligned}
\]

Likewise, we have $\rank^{\odot}(u'_1, u'_2, \ldots, u'_{r'}) \ge \rank^{\odot}(u_1, u_2, \ldots, u_{r})$. Taking the two inequalities together leads to an identity between the two ranks.

\end{proof}

\begin{proof}[Proof of Lemma \ref{lm:rank_identity}]
As earlier in this section, we can write $P$ as
\[
P \odot P = \sum_{1 \le i, j \le r} (u_i \odot u_j) (u_i \odot u_j)^\top.
\]
Let $R$ be an $n \times r^2$ matrix formed by the $r^2$ columns $u_i \odot u_j$ for $1 \le i, j \le n$. Clearly, $\rank(P \odot P) = \rank(R)$ since $P \odot P = R R^\top$. The (column) rank of $R$ is  $\rank^{\odot}(u_1, \ldots, u_r)$ by Definition \ref{def:o_rank}, as $u_i \odot u_j = u_j \odot u_i$). Hence, $\rank(P \odot P) = \rank^{\odot} (u_1, \ldots, u_r)$.

\end{proof}

\section{Rate of Convergence}
\label{rate}
We now turn to studying the rates of convergence of estimators analyzed in this paper.
For this, we need to introduce some additional notation.
We use $O(\cdot)$ and $o(\cdot)$ for the standard big-O and little-o notation. 
For two positive sequences $(a_n)_{n\ge 1}$, $(b_n)_{n\ge 1}$, we say $a_n = \Omega(b_n)$ if there exists a universal positive constant $c$ such that $a_n/b_n>c$.
The condition number of a square matrix $B$ is denoted by $\kappa(B)$. Convergence in probability and convergence in distribution are denoted as $\to_P$ and $\Rightarrow$, respectively.

We next give a fundamental result characterizing the sampling properties of the Hadamard estimator. This result bounds the relative error for estimating the vector of variances of all the entries of the OLS estimator. 
It shows that the estimation error is smaller when the aspect ratio $\gamma$ is small. 
We write $\vec\Sigma$ for the vector of the diagonal elements of $\Sigma$. 

\begin{theorem}[Rate of convergence]
Under the conditions of Theorem \ref{thm:Tnorm}, assume in addition that  the fourth moment of
 the entries $\ep_i$, $i\in[n]$
is less than a constant $C\ge 3$ times the  squared variance of the entries.
Let $V$ denote the vector of variances of the entries of the OLS estimator. Then, under high-dimensional asymptotics as $n,p\to\infty$ such that $\limsup \gamma_{p,n}=\limsup p/n < 1/2$, we have 
for any constant $c>1$ and some constant $C'>1$ that 
for all $n$ large enough,

$$ \mathbb{P} \left(\frac{\|\hV - V\|}{\|\vec\Sigma\|} \ge \frac{t}{n} \right) \le  \frac{2c}{t^2}\frac{1}{\left[\sigma_{\min}(\Gamma)(1-\gamma_{p,n}^{1/2})^2 (1-2 \gamma_{p,n})\right]^2} + C' n^{-1-\delta/4}.
$$
\label{r}
\end{theorem}
See Section \ref{pf:r} for a proof.

\subsection{Asymptotic normality}
\label{sec:norm}

We already know that the estimator $\hV$ is unbiased for the variances of the coordinates of the OLS  estimator $V = \diag\Cov{\hbeta}$, and in the previous section we have seen an inequality bounding the error $\|\hV-V\|$. 
In this section, we aim to study an estimator of
$w_p^\top S \Sigma S^\top w_p$ for a sequence $(w_p)_{p\ge 1}$  of vectors $w_p\in \mathbb{R}^p$.
This represents the variance of $w_p^\top \hbeta$, 
and taking $w_p$ to be the $i$-th canonical basis vectors in $\R^p$, for all $p$,
it reduces to $V_i$. The analysis will later be further used to derive an inferential method for $w_p^\top \beta$.

We use the coordinate-wise case of estimating $V_i$ for some $i\in [p]$, to illustrate the idea. 
To study the asymptotic distribution of $\hV_i = A_i^\top (\hep\odot\hep)$, where $A_i^\top$ is the $i$-th row of $A =  (S\odot S) (Q\odot Q)^{-1}$, we consider the noise $\ep = \Sigma^{1/2}Z$ to be linear combination of a vector of sufficiently smooth functions $Z$ of a
Gaussian random vector, specified by the conditions in Theorem \ref{thmhatviasym} below.
In that case, we can express the residuals as $\hep = Q\Sigma^{1/2}Z$. 

Thus, we see that the estimator $\hV_i$, a linear combination of squared entries of $\hep_i$, can be written as a symmetric quadratic form in $Z$. 
In particular, if $Z \sim N(0,I_n)$, 
its distribution is a weighted linear combination of chi-squared random variables.
For general $w_p$, the above discussion still applies by replacing $A_i^\top$ with $[(w_p^\top S) \odot (w_p^\top S)](Q\odot Q)^{-1}$.
For a linear combination of chi-squared random variables, we expect that it is close to a normal distribution
if none of the weights is too large. 
This is formalized by a so-called second order Poincar\'e inequality \citep{chatterjee2009fluctuations}. We will use this result to obtain 
the approximation to the normality of the variance estimator, given in the following result, proved in Section \ref{sec:pfhatvnormal}.
Let $d_{\mathrm{TV}}$ denote total variation distance.
For a function $f:\R\to\R$, $\|f\|_\infty  = \sup_{x\in\R}|f(x)|$ is the sup-norm.

\begin{theorem} [Approximate normality] \label{thmhatviasym} Assume $\epsilon = \Sigma^{1/2} Z$ 
where $Z = (Z_1, \ldots, Z_n)^\top$ 
consists of independent entries that have means zero, variances one, and fourth moments bounded by  $C_4$.
Further suppose 
that for all $i\in [n]$,
$Z_i = \psi_i(N_i)$, where $(N_1,\ldots N_n)^\top \sim \N(0,I_n)$, 
and $\psi_i:\R\to\R$ are twice-differentiable functions such that 
for all $i\in [n]$,
$\|\psi_i'\|_\infty < c_1 $ and $\|\psi_i''\|_\infty < c_2$ for some positive constants $c_1, c_2$.

Let $\widehat{\vec\Sigma}=(Q\odot Q)^{-1}(\hep \odot \hep).$
For a 
sequence of vectors 
$(w_p)_{p\ge 1}$, where  $w_p\in \mathbb{R}^p$ for all $p\ge 1$, 
of unit norm, we have  that
\begin{equation}\label{tvbd}
     d_{\mathrm{TV}}\left(\frac{w_p^\top S (\diag\widehat{\vec\Sigma}) S^\top w_p- w_p^\top S\Sigma S^\top w_p}{\sqrt{\mathrm{Var}{[w_p^\top S (\diag\widehat{\vec\Sigma}) S^\top w_p] }}},  \N(0,1)\right) \leq C_0\frac{\|G(w_p)\|}{\|G(w_p)\|_{\Fr}},
\end{equation}
where 
\begin{equation}\label{gw}
G(w_p) = \Sigma^{1/2}Q \diag\left\{[(w_p^\top S) \odot (w_p^\top S)](Q\odot Q)^{-1}\right\} Q \Sigma^{1/2}    
\end{equation}
and 
$C_0=\min\{C_4-1,2\}^{-1}\cdot$ $8\sqrt{5}c_1(4C_4+2)^{1/4}[(4C_4+2)^{1/4} c_2+ c_1^2]$.
\end{theorem}

We do not necessarily require $Z_i$, $i\in[p]$, to have unit variances, 
since we can normalize them and absorb the constants into $\Sigma$.
It is also possible to allow $C_4$, $c_1$ and $c_2$ to grow to infinity, which allows for some heavy-tailed distributions for $Z_i$.
In addition, as a special case of this theorem, taking $w_p$ to be the $i$-th canonical basis vectors in $\R^p$, for all $p$, leads to 
 \begin{equation*}
     d_{\mathrm{TV}}\left(\frac{\hat{V}_i- V_i}{\sqrt{\Var{\hat{V}_i}}},  \N(0,1)\right) \leq C_0\frac{\|G(e_i)\|}{\|G(e_i)\|_{\Fr}},
\end{equation*}
where  $G(e_i) = \Sigma^{1/2} Q\diag[e_i^\top (S\odot S) (Q\odot Q)^{-1}]Q\Sigma^{1/2}.$

In principle, this result could justify using normal confidence intervals for inference on $V_i$ as soon the upper bound provided is small. Moreover, the upper bound in Theorem \ref{thmhatviasym} can be simplified as follows. Denote $A(w_p) = [(w_p^\top S) \odot (w_p^\top S)](Q\odot Q)^{-1}$. We have the upper bound $\|G(w_p)\| \le \|\Sigma\| 
 \|Q \diag [A(w_p)] Q\|$ and the lower bound
 $ \|G(w_p)\|_{\Fr} \ge  \lambda_{\min}(\Sigma)  \| Q \diag[A(w_p)] Q\|_{\Fr}.$
 Therefore,  the upper bound simplifies to
 $$ C_0 \kappa(\Sigma) \frac{\| Q\diag[A(w_p)]Q \|}{ \| Q\diag[A(w_p)]Q \|_{\Fr}}.$$
This bound decouples 
as the product of a term depending on the unknown covariance matrix $\Sigma$, and the known design matrix $X$. Therefore, in practice one can evaluate the second term.
Thus, the deviation from normality only 
depends on the unknown $\Sigma$ through its condition number.

The subsequent proposition further characterizes conditions on the design matrix 
$X$ that lead to $\|G(w_p)\|/\|G(w_p)\|_{\Fr}\to 0$, thereby resulting in the asymptotic normality of the variance estimator of $w_p^\top S \Sigma S^\top w_p$. 
We also verify that the random design matrix specified in Theorem \ref{thm:Tnorm} satisfies these conditions with probability tending to one. 
See Section \ref{sec:pflemWnorm} for the proof.

\begin{proposition}[Conditions for asymptotic normality]\label{lemWnorm}
For $j\in[n]$,
let $S_{.j}$ be the $j$-th column of $S= (X^\top X)^{-1} X^\top$.
Consider the following conditions, where $c$ and $C$ are two positive constants: 
\begin{enumerate}
\item \label{cond1lemWnorm}
$\lambda_{\min}(Q\odot Q)>c$.
\item \label{cond2lemWnorm}
$\max_j |w_p^\top S_{.j}|=o\left(n^{-1/4}[\lambda_{\min}(X^\top X)]^{-1/2}\right).$
\item \label{cond3lemWnorm} $c<\lambda_{\min}(\Sigma)<\lambda_{\max}(\Sigma)<C.$
\end{enumerate}
Then, for $G(w_p)$ from \eqref{gw},
we have 
$$\|G(w_p)\| = o\left(n^{-1/2}[\lambda_{\min}(X^\top X)]^{-1}\right)\, \textnormal{ and  }\, \|G(w_p)\|_{\Fr} = \Omega\left(n^{-1/2}[\lambda_{\min}(X^\top X)]^{-1}\right).$$
In particular, $\|G(w_p)\|/\|G(w_p)\|_{\Fr}=o(1)$ and thus the total variation from \eqref{tvbd} vanishes asymptotically.
Moreover, 
for a random design matrix 
satisfying the conditions in Theorem \ref{thm:Tnorm}, 
if $\kappa(\Gamma)$ is bounded, the above conditions hold with probability tending to one; and thus $w_p^\top S (\diag\widehat{\vec\Sigma}) S^\top w_p$ is asymptotically normal.
\end{proposition}

After this detailed analysis of $w_p^\top S \Sigma S^\top w_p$ and its estimator $w_p^\top S (\diag\widehat{\vec\Sigma}) S^\top w_p$, we discuss inference for $w_p^\top \beta$. 
This task crucially relies on the first conclusion in the following ratio-consistency lemma, 
which is a simple consequence of several bounds obtained in Proposition \ref{lemWnorm}. The second conclusion in this lemma shows that although $\hV_i$---the unbiased estimator of $V_i$---can be negative with a small probability, all $\hV_i$, for $i\in[n]$, 
are simultaneously positive with probability tending to one. 
This finding is consistent with our numerical experiments.  
The proof of Lemma \ref{lem:cons-est} is in Section \ref{sec:pfinf}.
\begin{lemma}[Ratio-consistency]\label{lem:cons-est} We have the following ratio-consistency results:

\begin{enumerate} 
    \item \label{consistence1} Under the conditions of Theorem \ref{thmhatviasym} on the noise $\epsilon$, and conditions 1, 2, 3 from Proposition \ref{lemWnorm} on the data matrix $X$, we have $w_p^\top S (\diag\widehat{\vec\Sigma}) S^\top w_p/$ $w_p^\top S \Sigma S^\top w_p$ $\to_P 1$.
    \item If we further have $\max_{1\le i \le n}\E|Z_i|^8 = O(n)$, then it follows that
    $\max_{i=1}^n |\hV_i/V_i-1| = o_P(1).$
\end{enumerate}
\end{lemma}

The following result provides an inferential method for contrasts $w_p^\top \beta$. 
Its proof is built on Lemma \ref{lem:cons-est}, and is included in Section \ref{sec:pfinf}.

\begin{theorem}[Inference for contrasts]\label{thmbetanormal}
Under the conditions of 
claim \ref{consistence1} of Lemma \ref{lem:cons-est}, and 
further assuming that, with $S = (X^\top X)^{-1} X^\top$ ,
 \begin{equation}\label{assdelo}
    \frac{\max_{j\in[n]}|w_p^\top S\Sigma^{1/2} e_j|^2}{w_p^\top S \Sigma S^\top w_p}\to 0,
\end{equation}
we have 
$(w_p^\top \hat{\beta} - w_p^\top \beta)/\sqrt{w_p^\top S \diag\left[(Q\odot Q)^{-1}(\hep \odot \hep)\right] S^\top w_p} \Rightarrow \N(0,1)$.
\end{theorem}
Condition \ref{assdelo} can be viewed as a delocalization property, where none of the 
coordinates of the vector
$ \Sigma^{1/2} S^\top w_p$
dominate.

\section{Numerical Results}
\label{sec:num}

In this section, we present several numerical simulations supporting our theoretical results.
We consider the following cases:

\begin{enumerate}
    \item[] Case 1:  Take $X$ to have i.i.d.~standard normal entries, and the noise to be $\ep  = \Sigma^{1/2} Z$, where $Z$ has i.i.d.~standard normal entries. The noise covariance matrix $\Sigma$ is the diagonal matrix of eigenvalues of an AR-1 covariance matrix $T$, with $T_{ij} = \rho^{i-j}$. 
    \item[] Case 2: Take $X$ to have i.i.d.~$t_{10}$ entries, and the noise to be $\ep = \Sigma^{1/2} Z$ where $Z$ has i.i.d.~standard normal entries. The noise covariance matrix $\Sigma$ is the diagonal matrix consisting of the coordinates of $|c^\top X|$ where $c = (1,0, \ldots, 0)\in \mathbb{R}^n$.
\end{enumerate}

\subsection{Mean type I error over all coordinates} 

We show the mean type I error of the normal confidence intervals based on the White, MacKinnon-White, and Hadamard methods over all coordinates of the OLS estimator. 

\begin{figure}[ht!]
\centering
\begin{tabular}{ccc}
& {\large $\rho=0$} & {\large $\rho=0.9$} \\
{\begin{sideways}\parbox{\PBW\columnwidth}{\centering $\gamma = 0.1$}\end{sideways}} &
\includegraphics[width=\FW, trim = \TRA mm \TRB mm \TRC mm \TRD mm, clip = TRUE]{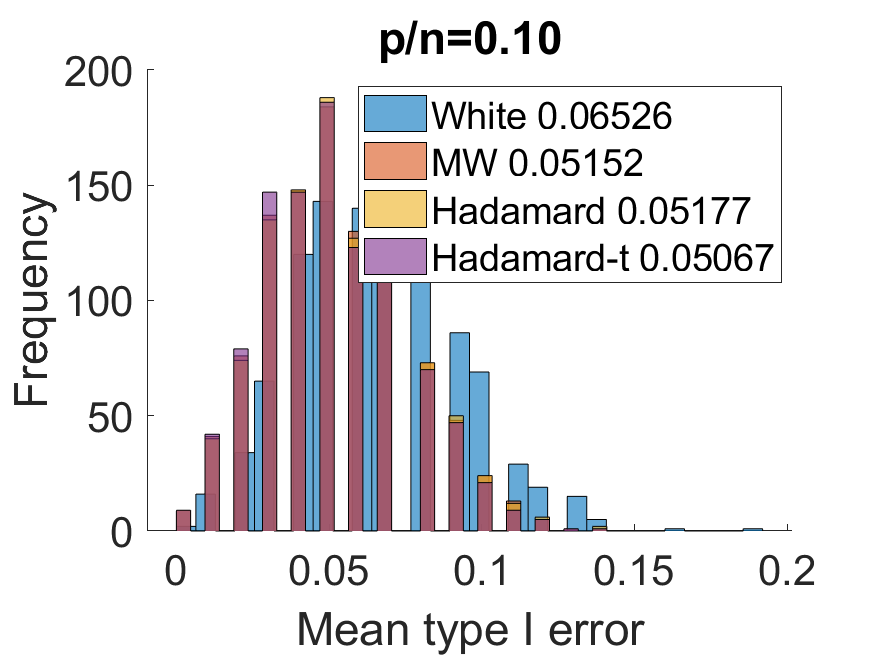} &
\includegraphics[width=\FW, trim = \TRA mm \TRB mm \TRC mm \TRD mm, clip = TRUE]{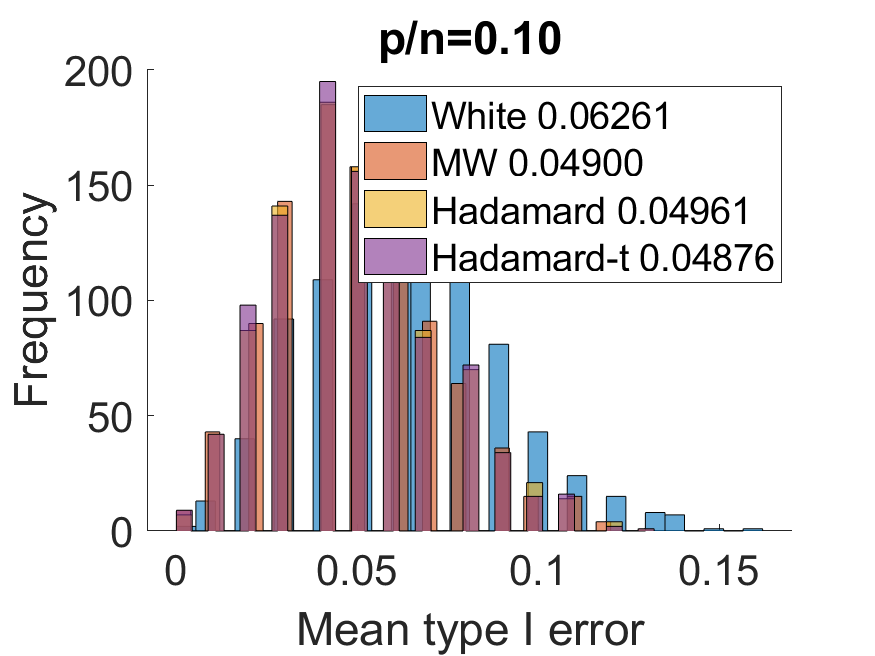} \\
{\begin{sideways}\parbox{\PBW\columnwidth}{\centering $\gamma = 0.5$}\end{sideways}} &
\includegraphics[width=\FW, trim = \TRA mm \TRB mm \TRC mm \TRD mm, clip = TRUE]{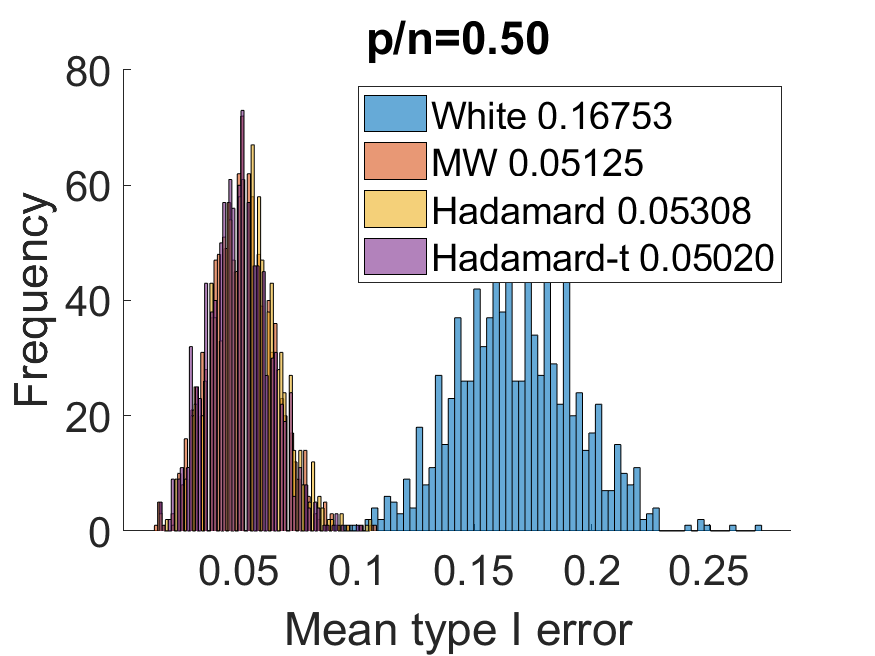} &
\includegraphics[width=\FW, trim = \TRA mm \TRB mm \TRC mm \TRD mm, clip = TRUE]{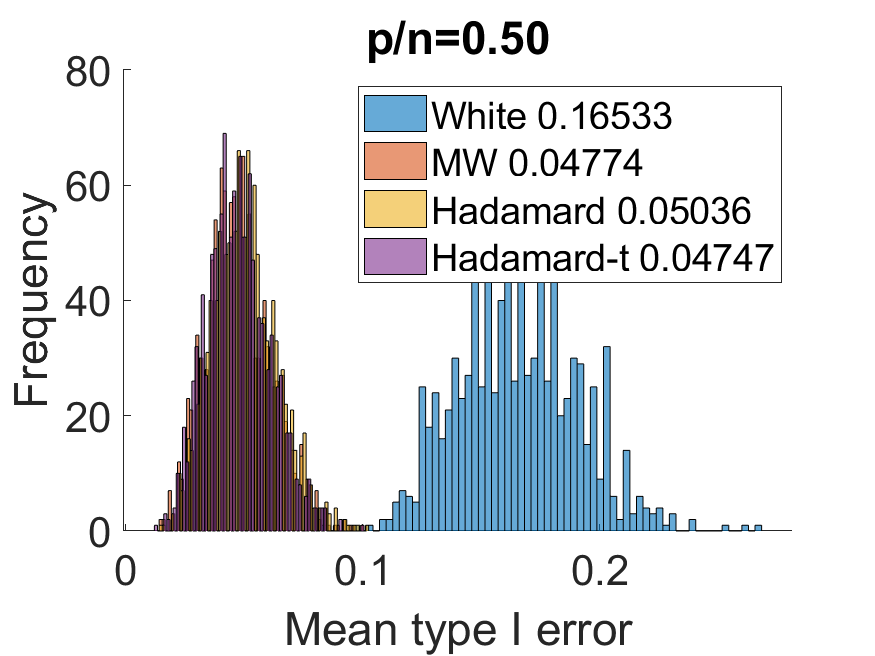}  \\
{\begin{sideways}\parbox{\PBW\columnwidth}{\centering $\gamma = 0.75$}\end{sideways}} &
\includegraphics[width=\FW, trim = \TRA mm \TRB mm \TRC mm \TRD mm, clip = TRUE]{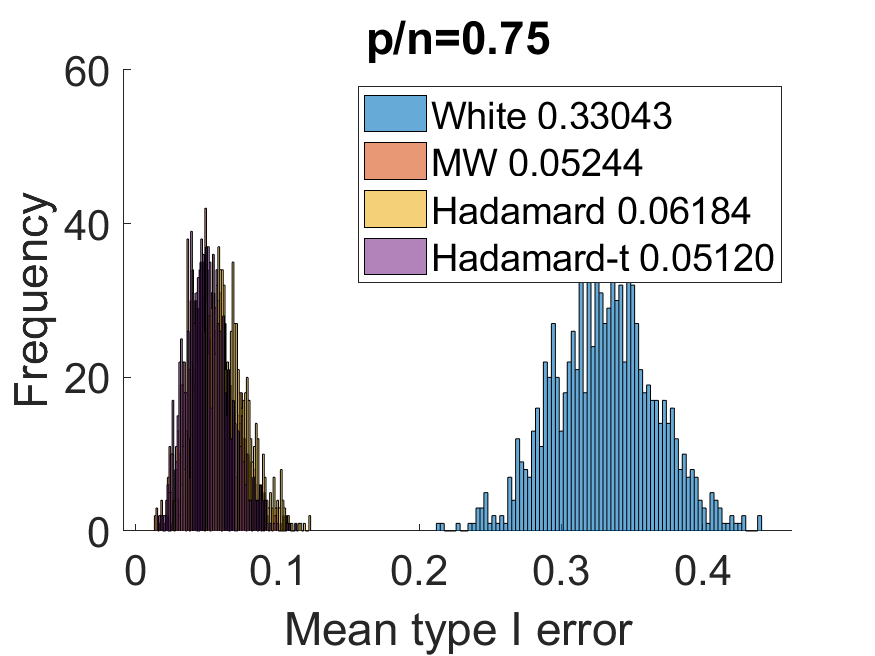} &
\includegraphics[width=\FW, trim = \TRA mm \TRB mm \TRC mm \TRD mm, clip = TRUE]{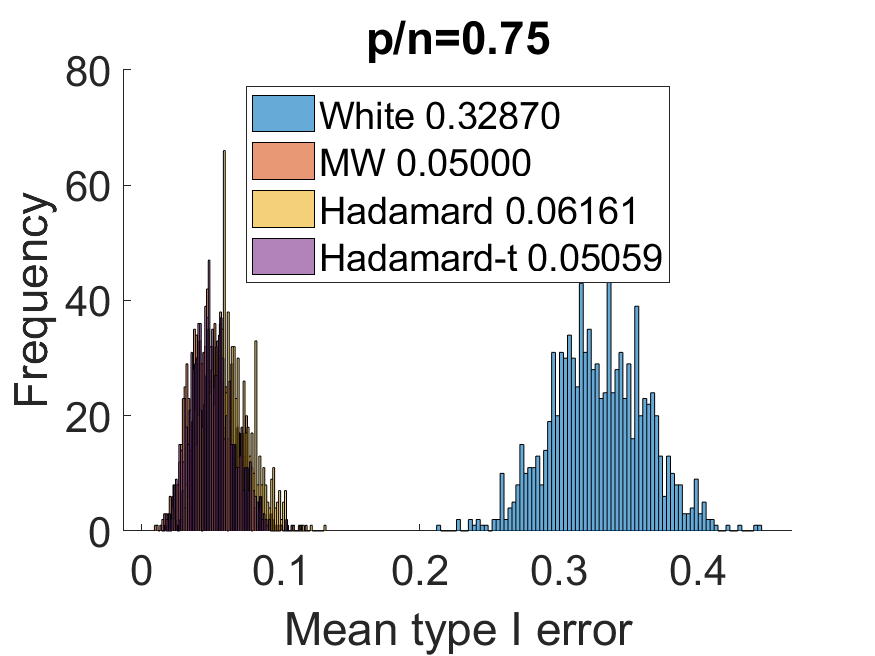} \\
\end{tabular}
\caption{Mean type I error over all coordinates.}
\label{fig:3}
\end{figure}

In Figure \ref{fig:3}, we show the results for Case 1.
We take $n=1000$,  and three aspect ratios, $\gamma = 0.1, 0.5, 0.75$, varying $p$. We consider $\rho=0$ (homoskedasticity), and $\rho=0.9$ (heteroskedasticity). We draw one instance of $X$, and draw 1000 Monte Carlo repetitions of $\ep$.

We observe that the CIs based on White's covariance matrix estimator are inaccurate for the aspect ratios considered. They have inflated type I error rates. All other estimators are more accurate. The MW confidence intervals are quite accurate for each configuration. The Hadamard estimator using the degrees of freedom correction is comparable, and noticably better if the dimension is high. 

\subsection{Coordinate-wise type I error} 

The situation is more nuanced, however, when we look at individual coordinates. In Table \ref{my-label}, we report the empirical type I error of the methods for the first coordinate in Case 1, where the average is over the Monte Carlo trials. 
In this case, the MW estimator can be either liberal or conservative, while the Hadamard estimator is closer to having the right coverage.

\begin{figure}[ht!]
    \begin{subfigure}{.5\textwidth}
  \centering
  \includegraphics[scale=0.5]{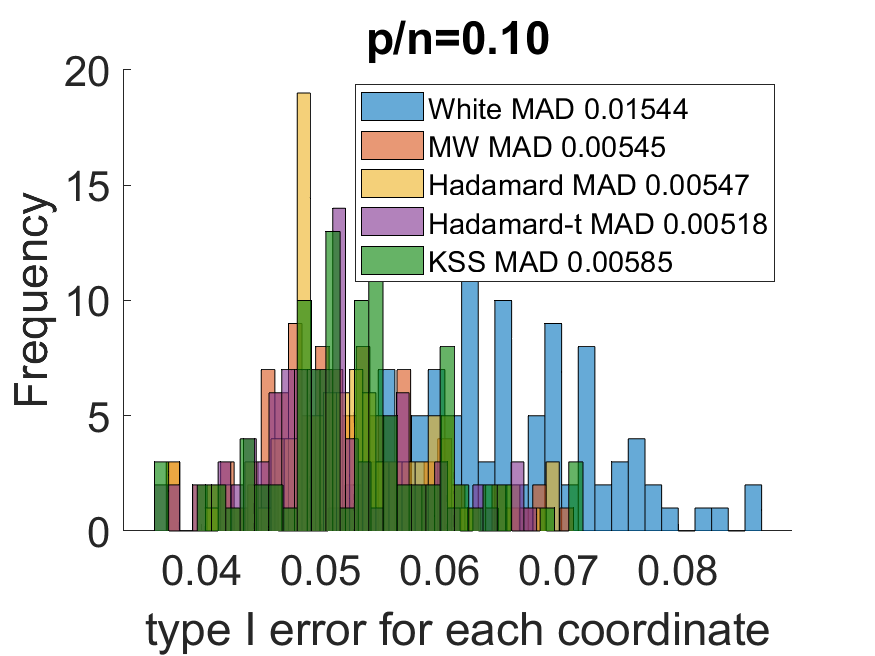}
\end{subfigure}
\begin{subfigure}{0.5\textwidth}
  \centering
  \includegraphics[scale=0.5]{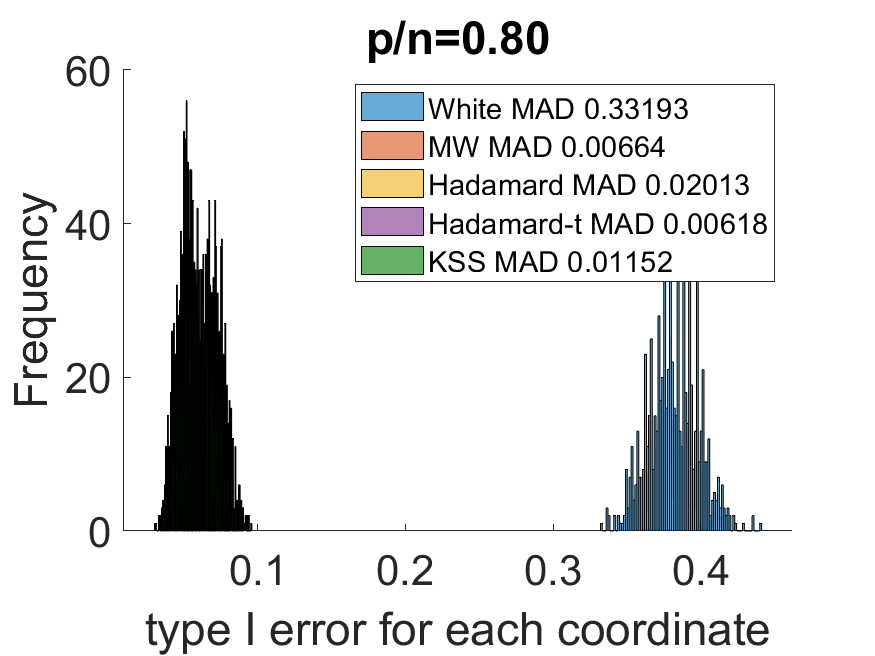}
\end{subfigure}
\caption{Mean type-I error for each coordinate over 1000 simulations. }
    \label{fig:allcoordinae}
\end{figure}

\renewcommand{\arraystretch}{1.5}
\begin{table}
\centering
\caption{Type I error for the first coordinate.}
\label{my-label}
\begin{tabular}{|l|l|l|l|l|}
\hline
$\gamma = p/n$ & White            & MW  & Hadamard   & Hadamard-t \\ \hline
0.5         & 0.172 &0.045 &0.042 &0.039\\ \hline
0.75     & 0.347 &0.059 &0.053 &0.047 \\ \hline
\end{tabular}
\end{table}

Figure \ref{fig:allcoordinae} displays the mean type I error of each coordinate in Case 2. 
We also compare with the leave out (KSS) estimator developed in \cite{kline2020leave}.
The results for Case 1 are in Section \ref{sec:appendix-simulations} of the appendix.
To measure the overall accuracy, we also report the mean absolute deviation
$\mathrm{MAD} = p^{-1}\sum_{j=1}^p |\mathrm{err}_j - 0.05|$ where $\mathrm{err}_j$ refers to the type I error for the $j$-th coordinate. 
The White estimator has inflated type I errors especially for larger $p$.  We also observe that although the MAD of the Hadamard estimator is large when $p=800$, the degrees-of-freedom adjustment significantly improves the performance and achieves performance comparable  to the MW estimator.
The performance of the KSS estimator resembles that of the Hadamard-t estimator.

To further compare with other methods, 
we plot the mean type I error as a function of $p$ by taking $p$ equally spaced from 100 to 800 with gaps of 100, see Figure \ref{fig:case2-bootstrap}. 
The results including the KSS estimator and for Case 1 are reported in Section \ref{sec:appendix-simulations} of the appendix. 
We observe that the  MW estimators are liberal for the first coordinate but accurate for the second coordinate. The CI based on the Hadamard estimator has a slightly inflated type I error for both coordinates for larger $p$, but the Hadamard-t estimator is accurate.  

For comparison, we also conduct experiments using two variants of the bootstrap.
First, we use the pairs bootstrap
\citep{freedman1981bootstrapping}, 
where each observation of the bootstrap sample $[X^*, Y^*]$ is sampled randomly with replacement
from the rows of $[X, Y]$, see Figure \ref{fig:case2-bootstrap}.
We also include the residual bootstrap,
which samples with replacement the residuals $Y - X\hat{\beta}$, and adds them to $X\hat{\beta}$ to get the new $Y^*$. 
Intuitively, this method is justified if the error terms are independent and identically distributed (see e.g., \cite{mackinnon2006bootstrap} for a discussion). 
The corresponding results are in Figure \ref{fig:case2-bootstrap-6methods}. 

Since the rows are sampled with replacement 
for the pairs bootstrap, 
when $p$ is close to $n$,  the resampled $X^{*\top}X^*$ may be ill-conditioned or non-invertible.
Figure \ref{fig:case2-bootstrap} shows the results where the matrix inversion is done using a pseudoinverse when $p>500$. 
This leads to unstable confidence intervals, coverage and length. 
The lengths of the confidence intervals are in the right panel of Figure \ref{fig:case2-bootstrap-6methods}.

We also conduct experiments using the jackknife (or equivalently HC3 in \cite{mackinnon1985some}), and 
see from Figures \ref{fig:case1-bootstrap-jackknife-first} and \ref{fig:case2-bootstrap-jackknife-first} that the jackknife is not accurate. 
This can be explained by the fact that the expression for the jackknife estimator therein was derived for a relatively small dimension $p$.

\begin{figure}[ht!]
\begin{subfigure}{.5\textwidth}
  \centering
  \includegraphics[scale=0.5]{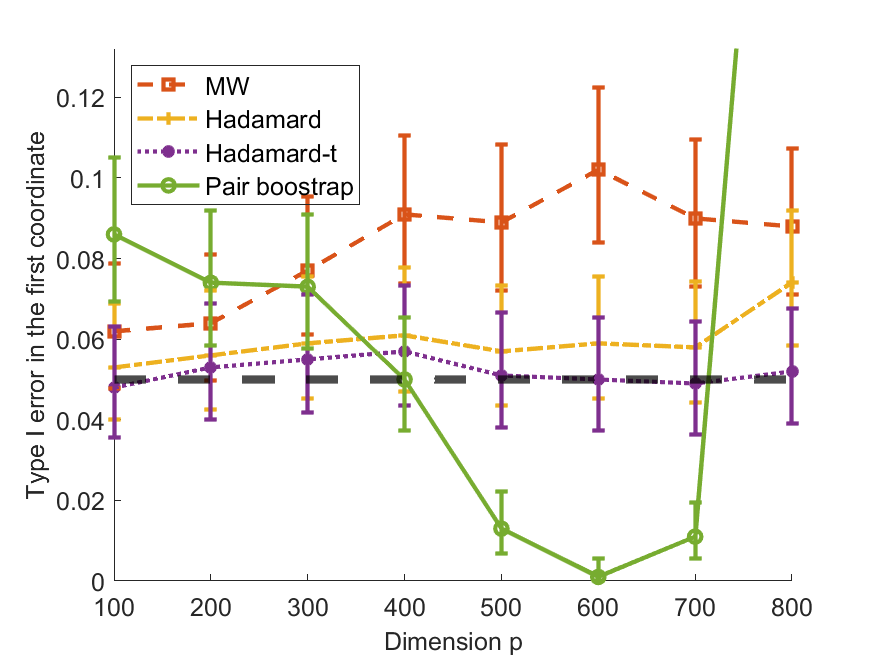}
\end{subfigure}
\begin{subfigure}{0.5\textwidth}
  \centering
  \includegraphics[scale=0.5]{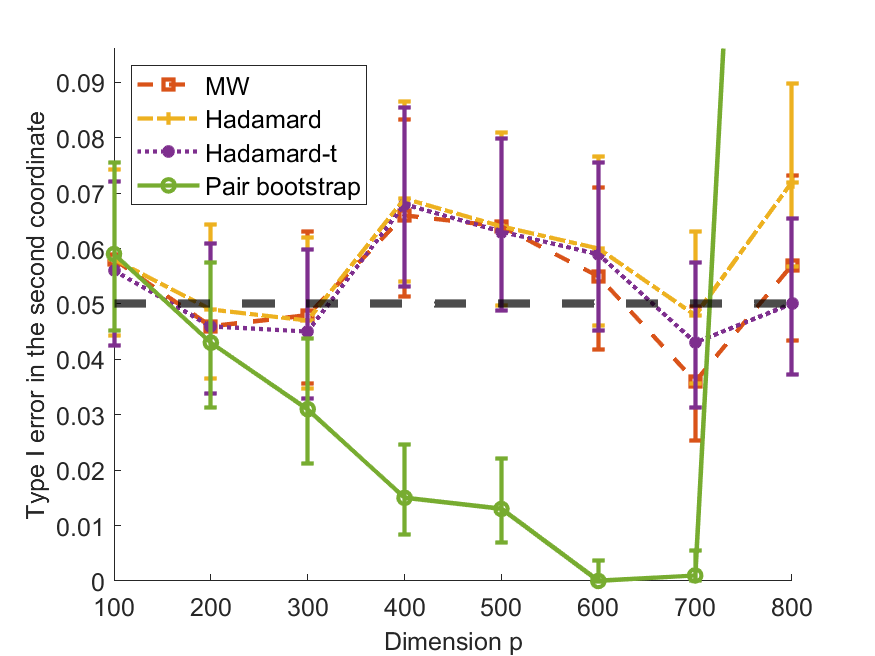}
\end{subfigure}
\caption{Mean type-I error in the first coordinate and second coordinate over 1000 simulations each. The error bars represent 95\% Clopper-Pearson intervals for the coverage.}
    \label{fig:case2-bootstrap}
\end{figure}

\subsection{Estimating the MSE} 

\begin{figure}[ht!]
\begin{subfigure}{.5\textwidth}
  \centering
  \includegraphics[scale=0.45]{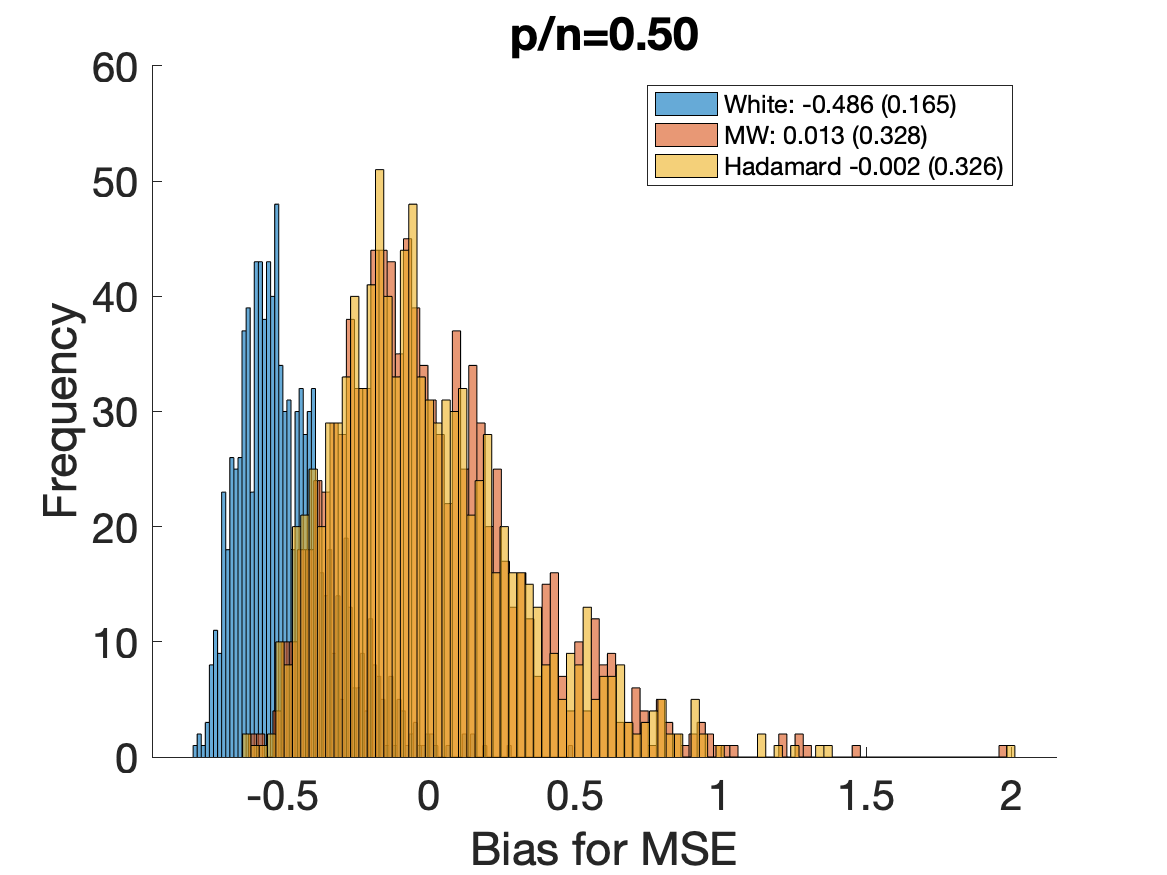}
  \caption{$p/n=0.5$}
\end{subfigure}
\begin{subfigure}{.5\textwidth}
  \centering
  \includegraphics[scale=0.45]{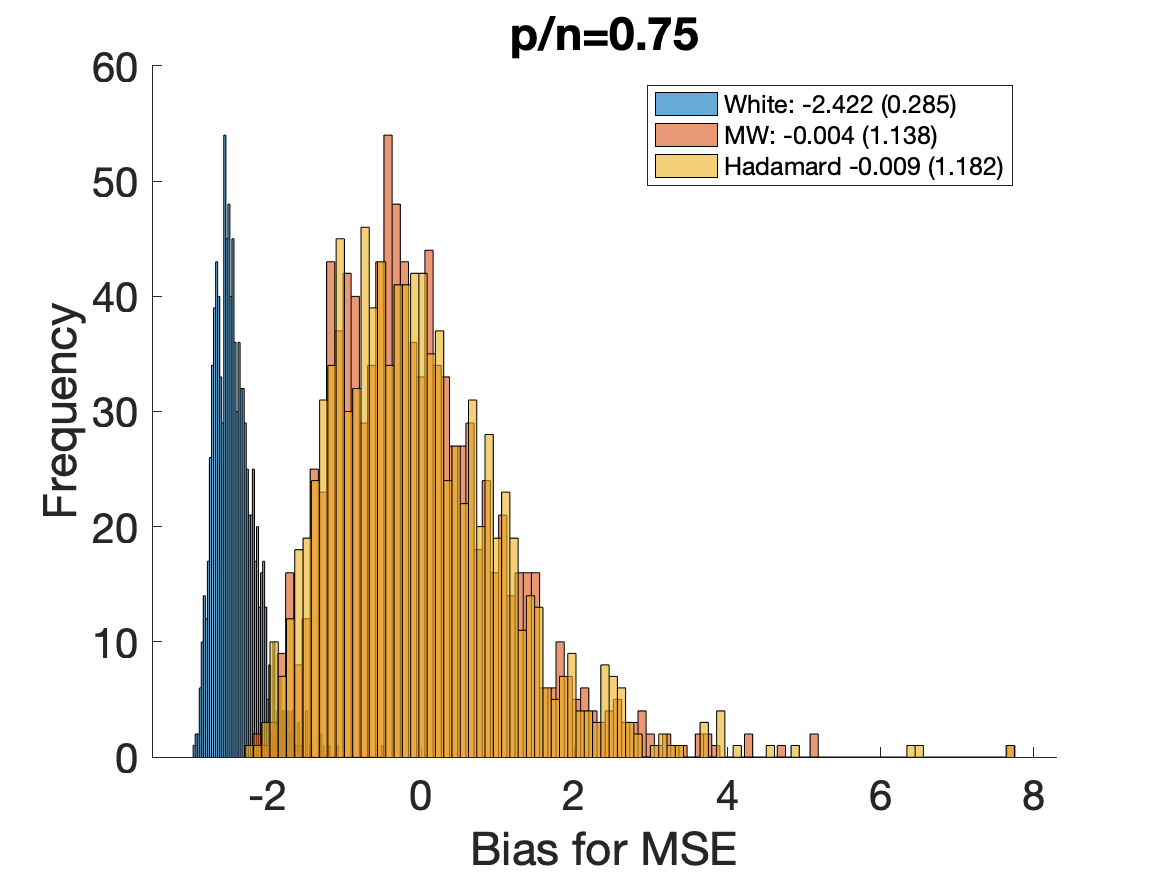}
  \caption{$p/n=0.75$}
\end{subfigure}
\caption{Bias in estimating MSE.}
\label{fig:4}
\end{figure}

In Figure \ref{fig:4}, we show the bias in estimating the MSE of the OLS estimator in Case 1 for the three methods where the numbers outside the bracket correspond to the mean biases of the 1000 Monte Carlo runs and the numbers inside the bracket stand for the standard deviation. For each method, we use the estimator which equals the sum of the variances of the individual component estimators. 

The results are in line with those from the previous sections. Both MacKinnon-White's and the Hadamard estimator perform much better than the White estimator. In addition, the Hadamard estimator is usually comparable to MacKinnon-White's. More results for Case 2 are included in Section \ref{sec:appendix-simulations} of the Appendix.


\subsection{Approximate Normality} 
\label{ano}
\begin{figure}[ht!]
\begin{subfigure}{.5\textwidth}
  \centering
  \includegraphics[scale=0.5]{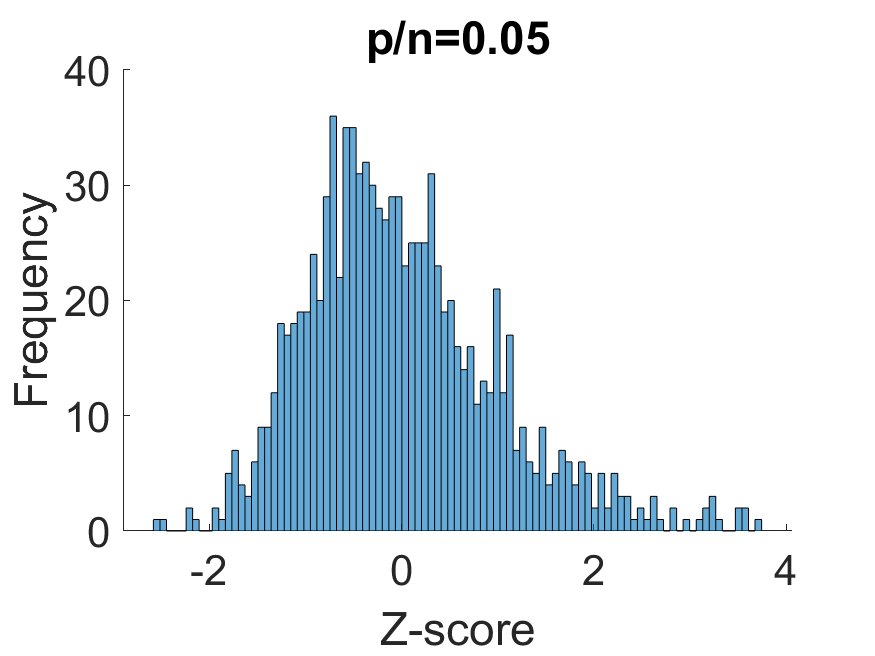}
  \caption{$p/n=0.05$}
\end{subfigure}
\begin{subfigure}{.5\textwidth}
  \centering
  \includegraphics[scale=0.5]{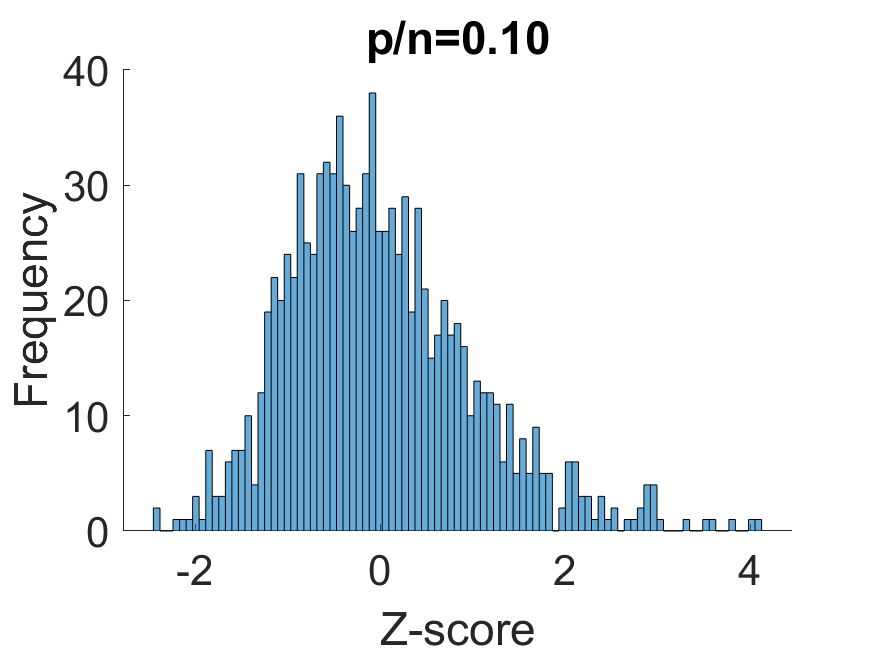}
  \caption{$p/n=0.1$}
\end{subfigure}
\caption{Distribution of $z$-scores of a fixed coordinate of the Hadamard estimator.}
\label{fig:5}
\end{figure}

In Figure \ref{fig:5}, we show the distribution of $z$-scores of a fixed coordinate of the Hadamard estimator in Case 1. We use a similar setup to the previous sections, but we choose a larger sample size $n=1,000$, and also smaller aspect ratios $p/n=0.05$ and $p/n=0.1$. We observe a relatively good fit to the normal distribution, but it is also clear that a chi-squared approximation may lead to a better fit.


\section{Discussion and Future Work}


In this paper, we have  provided several fundamental theoretical results for 
the Hadamard estimator,
which can be used to construct confidence intervals for the OLS estimator of coefficients in a linear model under heteroskedasticity.
We showed that the Hadamard estimator is well-defined and well-conditioned for certain random design models. There are several  important directions for future research.
Can one develop similar results for nonlinear models? 
Is it possible to establish the
non-negativity of the Hadamard estimator, possibly with some
regularization? Is it possible to show approximate coverage results
for our $t$-confidence intervals based on the degrees of freedom correction as given in \eqref{eq:d_formula}? Such results have been obtained in the low-dimensional case by \cite{kauermann2001note}, for instance. However, establishing such results in high dimensions seems to require different techniques.

Beyond our current investigations, an important direction is the development of tests for heteroskedasticity. White's original paper proposed such a test based on comparing his covariance estimator to the usual one under homoskedasticity. There are many other well-known proposals \citep{dette1998testing,
azzalini1993use,
cook1983diagnostics,
breusch1979simple, wang2018empirical}. Perhaps most closely related to our work, 
\cite{li2019testing} have proposed tests for heteroskedasticity with good properties in low and high dimensional settings. Their tests rely on computing measures of variability of the estimated residuals, including the ratio of the arithmetic and geometric means, as well as the coefficient of variation. Their works and follow-ups such as  \cite{bai2016homoscedasticity, bai2017central} show central limit theorems for these test statistics. They also show an improved empirical power compared to some classical tests for heteroskedasticity. It would be of interest to see if our covariance matrix estimator could be used to develop new tests for heteroskedasticity.

An important extension of the heteroskedastic model is the clustered observations model. \cite{liang1986longitudinal} proposed estimating equations for such longitudinal/clustered data. They allowed arbitrarily correlated observations for any fixed individual (i.e., within each cluster), and proposed a consistent covariance estimator in the low-dimensional setting.  Can one  extend our ideas to the clustered case?

Another important direction is to develop covariance estimators that have good performance in the presence of both heteroskedasticity and autocorrelation. The most well-known example is possibly the popular Newey-West  estimator \citep{west1987simple}, which is a sum of symmetrized lagged autocovariance matrices with decaying weights. Is it possible to develop new methods inspired by our ideas suitable for this setting?

Our paper does not touch on the interesting but challenging regime where $n<p$. In that setting, Buhlmann, Dezeure, Zhang, \citep{dezeure2016high} proposed bootstrap methods for inference with the lasso under heteroskedasticity, under the limited ultra-sparse regime, where the sparsity $s$ of the regression parameter is $s \ll n^{1/2}$. These methods are limited as they apply only to the lasso, and because they only concern the ultra-sparse regime. It would be interesting to understand this regime better.

It is possible that one could sometimes get better confidence intervals by adding some regularization to the starting estimator, such as starting with an $\ell_2$ regularized regression estimator. However, note that this would come at the cost of introducing some bias, so that each confidence interval would not be centered at the true parameter anymore. It is possible that by an appropriately small regularization, one could achieve a favorable ``bias-variance" tradeoff for confidence intervals. However, the specific details are likely complex (e.g., how to tune the regularization) and deserve a separate investigation.

\section{Acknowledgements}

The authors thank Matias Cattaneo and Jason Klusowski  for valuable discussions and feedback on an earlier version of the manuscript. We are grateful to the referees, whose valuable suggestions have lead to many improvements to the paper.
This work was supported in part by the NSF grant DMS 2046874.

{\small
\setlength{\bibsep}{0.2pt plus 0.3ex}
\bibliographystyle{plainnat-abbrev}
\bibliography{references}

\begin{thebibliography}{44}
\providecommand{\natexlab}[1]{#1}
\providecommand{\url}[1]{\texttt{#1}}
\expandafter\ifx\csname urlstyle\endcsname\relax
  \providecommand{\doi}[1]{doi: #1}\else
  \providecommand{\doi}{doi: \begingroup \urlstyle{rm}\Url}\fi

\bibitem[Azzalini and Bowman(1993)]{azzalini1993use}
A.~Azzalini and A.~Bowman.
\newblock On the use of nonparametric regression for checking linear
  relationships.
\newblock \emph{Journal of the Royal Statistical Society. Series B
  (Methodological)}, 55\penalty0 (2):\penalty0 549--557, 1993.

\bibitem[Bai and Silverstein(2010)]{bai2010spectral}
Z.~Bai and J.~W. Silverstein.
\newblock \emph{Spectral analysis of large dimensional random matrices}.
\newblock Springer, 2010.

\bibitem[Bai et~al.(2016)Bai, Pan, and Yin]{bai2016homoscedasticity}
Z.~Bai, G.~Pan, and Y.~Yin.
\newblock Homoscedasticity tests for both low and high-dimensional fixed design
  regressions.
\newblock \emph{arXiv preprint arXiv:1603.03830}, 2016.

\bibitem[Bai et~al.(2018)Bai, Pan, and Yin]{bai2017central}
Z.~Bai, G.~Pan, and Y.~Yin.
\newblock A central limit theorem for sums of functions of residuals in a
  high-dimensional regression model with an application to variance
  homoscedasticity test.
\newblock \emph{Test}, 27:\penalty0 896--920, 2018.

\bibitem[Bell and McCaffrey(2002)]{bell2002bias}
R.~M. Bell and D.~F. McCaffrey.
\newblock Bias reduction in standard errors for linear regression with
  multi-stage samples.
\newblock \emph{Survey Methodology}, 28\penalty0 (2):\penalty0 169--182, 2002.

\bibitem[Bera et~al.(2002)Bera, Suprayitno, and Premaratne]{bera2002some}
A.~K. Bera, T.~Suprayitno, and G.~Premaratne.
\newblock On some heteroskedasticity-robust estimators of variance--covariance
  matrix of the least-squares estimators.
\newblock \emph{Journal of Statisticsistical Planning and Inference},
  108\penalty0 (1-2):\penalty0 121--136, 2002.

\bibitem[Breusch and Pagan(1979)]{breusch1979simple}
T.~S. Breusch and A.~R. Pagan.
\newblock A simple test for heteroscedasticity and random coefficient
  variation.
\newblock \emph{Econometrica: Journal of the Econometric Society}, 47\penalty0
  (5):\penalty0 1287--1294, 1979.

\bibitem[Cattaneo et~al.(2018)Cattaneo, Jansson, and
  Newey]{cattaneo2017inference}
M.~D. Cattaneo, M.~Jansson, and W.~K. Newey.
\newblock Inference in linear regression models with many covariates and
  heteroscedasticity.
\newblock \emph{Journal of the American Statistical Association}, 113\penalty0
  (523):\penalty0 1350--1361, 2018.

\bibitem[Chatterjee(2009)]{chatterjee2009fluctuations}
S.~Chatterjee.
\newblock Fluctuations of eigenvalues and second order poincar{\'e}
  inequalities.
\newblock \emph{Probability Theory and Related Fields}, 143\penalty0
  (1-2):\penalty0 1--40, 2009.

\bibitem[Chen et~al.(2016)Chen, Gao, and Ren]{chen2016general}
M.~Chen, C.~Gao, and Z.~Ren.
\newblock A general decision theory for huber’s epsilon-contamination model.
\newblock \emph{Electronic Journal of Statistics}, 10\penalty0 (2):\penalty0
  3752--3774, 2016.

\bibitem[Chew(1970)]{chew1970covariance}
V.~Chew.
\newblock Covariance matrix estimation in linear models.
\newblock \emph{Journal of the American Statistical Association}, 65\penalty0
  (329):\penalty0 173--181, 1970.

\bibitem[Cook and Weisberg(1983)]{cook1983diagnostics}
R.~D. Cook and S.~Weisberg.
\newblock Diagnostics for heteroscedasticity in regression.
\newblock \emph{Biometrika}, 70\penalty0 (1):\penalty0 1--10, 1983.

\bibitem[Dette and Munk(1998)]{dette1998testing}
H.~Dette and A.~Munk.
\newblock Testing heteroscedasticity in nonparametric regression.
\newblock \emph{Journal of the Royal Statistical Society: Series B (Statistical
  Methodology)}, 60\penalty0 (4):\penalty0 693--708, 1998.

\bibitem[Dezeure et~al.(2017)Dezeure, B{\"u}hlmann, and Zhang]{dezeure2016high}
R.~Dezeure, P.~B{\"u}hlmann, and C.-H. Zhang.
\newblock High-dimensional simultaneous inference with the bootstrap.
\newblock \emph{Test}, 26:\penalty0 685--719, 2017.

\bibitem[Diakonikolas et~al.(2017)Diakonikolas, Kamath, Kane, Li, Moitra, and
  Stewart]{diakonikolas2017being}
I.~Diakonikolas, G.~Kamath, D.~M. Kane, J.~Li, A.~Moitra, and A.~Stewart.
\newblock Being robust (in high dimensions) can be practical.
\newblock In \emph{International Conference on Machine Learning}, pages
  999--1008. PMLR, 2017.

\bibitem[Dicker(2014)]{dicker2014variance}
L.~H. Dicker.
\newblock Variance estimation in high-dimensional linear models.
\newblock \emph{Biometrika}, 101\penalty0 (2):\penalty0 269--284, 2014.

\bibitem[Donoho and Montanari(2016)]{donoho2016high}
D.~Donoho and A.~Montanari.
\newblock High dimensional robust m-estimation: Asymptotic variance via
  approximate message passing.
\newblock \emph{Probability Theory and Related Fields}, 166\penalty0
  (3-4):\penalty0 935--969, 2016.

\bibitem[Eicker(1967)]{eicker1967limit}
F.~Eicker.
\newblock Limit theorems for regressions with unequal and dependent errors.
\newblock In \emph{Proceedings of the fifth Berkeley symposium on mathematical
  statistics and probability}, volume~1, pages 59--82, 1967.

\bibitem[El~Karoui et~al.(2013)El~Karoui, Bean, Bickel, Lim, and
  Yu]{el2013robust}
N.~El~Karoui, D.~Bean, P.~J. Bickel, C.~Lim, and B.~Yu.
\newblock On robust regression with high-dimensional predictors.
\newblock \emph{Proceedings of the National Academy of Sciences}, 110\penalty0
  (36):\penalty0 14557--14562, 2013.

\bibitem[Freedman(1981)]{freedman1981bootstrapping}
D.~A. Freedman.
\newblock Bootstrapping regression models.
\newblock \emph{The Annals of Statistics}, 9\penalty0 (6):\penalty0 1218--1228,
  1981.

\bibitem[Greene(2003)]{greene2003econometric}
W.~H. Greene.
\newblock \emph{Econometric analysis}.
\newblock Pearson, 2003.

\bibitem[Hartley et~al.(1969)Hartley, Rao, and Kiefer]{hartley1969variance}
H.~Hartley, J.~Rao, and G.~Kiefer.
\newblock Variance estimation with one unit per stratum.
\newblock \emph{Journal of the American Statistical Association}, 64\penalty0
  (327):\penalty0 841--851, 1969.

\bibitem[Horn and Johnson(1990)]{horn1990matrix}
R.~A. Horn and C.~R. Johnson.
\newblock \emph{Matrix analysis}.
\newblock Cambridge University Press, 1990.

\bibitem[Horn and Johnson(1994)]{horn1991topics}
R.~A. Horn and C.~R. Johnson.
\newblock \emph{Topics in matrix analysis}.
\newblock Cambridge University Press, 1994.

\bibitem[Huber(1967)]{huber1967behavior}
P.~J. Huber.
\newblock The behavior of maximum likelihood estimates under nonstandard
  conditions.
\newblock In \emph{Proceedings of the fifth Berkeley symposium on mathematical
  statistics and probability}, volume~1, pages 221--233. Berkeley, CA, 1967.

\bibitem[Huber and Ronchetti(2011)]{huber2011robust}
P.~J. Huber and E.~M. Ronchetti.
\newblock \emph{Robust statistics}.
\newblock Wiley, 2011.

\bibitem[Imbens and Kolesar(2016)]{imbens2016robust}
G.~W. Imbens and M.~Kolesar.
\newblock Robust standard errors in small samples: Some practical advice.
\newblock \emph{Review of Economics and Statistics}, 98\penalty0 (4):\penalty0
  701--712, 2016.

\bibitem[Janson et~al.(2017)Janson, Barber, and Cand\`es]{janson2017eigenprism}
L.~Janson, R.~F. Barber, and E.~Cand\`es.
\newblock Eigenprism: inference for high dimensional signal-to-noise ratios.
\newblock \emph{Journal of the Royal Statistical Society: Series B (Statistical
  Methodology)}, 79\penalty0 (4):\penalty0 1037--1065, 2017.

\bibitem[Kauermann and Carroll(2001)]{kauermann2001note}
G.~Kauermann and R.~J. Carroll.
\newblock A note on the efficiency of sandwich covariance matrix estimation.
\newblock \emph{Journal of the American Statistical Association}, 96\penalty0
  (456):\penalty0 1387--1396, 2001.

\bibitem[Kline et~al.(2020)Kline, Saggio, and S{\o}lvsten]{kline2020leave}
P.~Kline, R.~Saggio, and M.~S{\o}lvsten.
\newblock Leave-out estimation of variance components.
\newblock \emph{Econometrica}, 88\penalty0 (5):\penalty0 1859--1898, 2020.

\bibitem[Lehmann and Casella(1998)]{lehmann1998theory}
E.~Lehmann and G.~Casella.
\newblock Theory of point estimation.
\newblock \emph{Springer Texts in Statistics}, 1998.

\bibitem[Li and Yao(2019)]{li2019testing}
Z.~Li and J.~Yao.
\newblock Testing for heteroscedasticity in high-dimensional regressions.
\newblock \emph{Econometrics and Statistics}, 9:\penalty0 122--139, 2019.

\bibitem[Liang and Zeger(1986)]{liang1986longitudinal}
K.-Y. Liang and S.~L. Zeger.
\newblock Longitudinal data analysis using generalized linear models.
\newblock \emph{Biometrika}, 73\penalty0 (1):\penalty0 13--22, 1986.

\bibitem[MacKinnon(2006)]{mackinnon2006bootstrap}
J.~G. MacKinnon.
\newblock Bootstrap methods in econometrics.
\newblock \emph{Economic Record}, 82:\penalty0 S2--S18, 2006.

\bibitem[MacKinnon and White(1985)]{mackinnon1985some}
J.~G. MacKinnon and H.~White.
\newblock Some heteroskedasticity-consistent covariance matrix estimators with
  improved finite sample properties.
\newblock \emph{Journal of econometrics}, 29\penalty0 (3):\penalty0 305--325,
  1985.

\bibitem[Paul and Aue(2014)]{paul2014random}
D.~Paul and A.~Aue.
\newblock Random matrix theory in statistics: A review.
\newblock \emph{Journal of Statisticsistical Planning and Inference},
  150:\penalty0 1--29, 2014.

\bibitem[Visscher et~al.(2008)Visscher, Hill, and
  Wray]{visscher2008heritability}
P.~M. Visscher, W.~G. Hill, and N.~R. Wray.
\newblock Heritability in the genomics era—concepts and misconceptions.
\newblock \emph{Nature reviews genetics}, 9\penalty0 (4):\penalty0 255--266,
  2008.

\bibitem[Wang et~al.(2018)Wang, Zhong, and Cui]{wang2018empirical}
H.~Wang, P.-S. Zhong, and Y.~Cui.
\newblock Empirical likelihood ratio tests for coefficients in high-dimensional
  heteroscedastic linear models.
\newblock \emph{Statistica Sinica}, 28\penalty0 (4):\penalty0 2409--2433, 2018.

\bibitem[West and Newey(1987)]{west1987simple}
K.~D. West and W.~K. Newey.
\newblock A simple, positive semi-definite, heteroskedasticity and
  autocorrelation consistent covariance matrix.
\newblock \emph{Econometrica}, 55\penalty0 (3):\penalty0 703--708, 1987.

\bibitem[White(1980)]{white1980heteroskedasticity}
H.~White.
\newblock A heteroskedasticity-consistent covariance matrix estimator and a
  direct test for heteroskedasticity.
\newblock \emph{Econometrica}, 48\penalty0 (4):\penalty0 817--838, 1980.

\bibitem[Wu(1986)]{wu1986jackknife}
C.-F.~J. Wu.
\newblock Jackknife, bootstrap and other resampling methods in regression
  analysis.
\newblock \emph{The Annals of Statistics}, 14\penalty0 (4):\penalty0
  1261--1295, 1986.

\bibitem[Yang(2020)]{yang2020linear}
F.~Yang.
\newblock Linear spectral statistics of eigenvectors of anisotropic sample
  covariance matrices.
\newblock \emph{arXiv preprint arXiv:2005.00999}, 2020.

\bibitem[Yao et~al.(2015)Yao, Bai, and Zheng]{yao2015large}
J.~Yao, Z.~Bai, and S.~Zheng.
\newblock \emph{Large Sample Covariance Matrices and High-Dimensional Data
  Analysis}.
\newblock Cambridge University Press, 2015.

\bibitem[Zhou et~al.(2018)Zhou, Bose, Fan, and Liu]{zhou2018new}
W.-X. Zhou, K.~Bose, J.~Fan, and H.~Liu.
\newblock A new perspective on robust {M}-estimation: Finite sample theory and
  applications to dependence-adjusted multiple testing.
\newblock \emph{Annals of Statistics}, 46\penalty0 (5):\penalty0 1904, 2018.

\end{thebibliography}
}

\appendix
\section*{Appendix}

\medskip
{\bf Notation.}
For two positive sequences $(a_n)_{n\ge 1}$, $(b_n)_{n\ge 1}$, 
we write $a_n \asymp b_n$ if $C^{-1}b_n\le a_n \le C b_n$ for some positive constant $C$.

\section{Proofs}

\subsection{Proof of unbiasedness of the Hadamard estimator}
\label{pf:unb}

We consider estimators of the vector of variances of $\hbeta$ of the form 
$\hV = A\cdot (\hep\odot\hep)$
where $A$ is a $p\times n$ matrix, and $M\odot M$ is the element-wise (or Hadamard) product of the vector or matrix $M$ with itself. Our goal is to find $A$ such that $\E \hV = V$, where $V = \diag \Cov{\hbeta}$. Here the $\diag$ operator returns the vector of diagonal entries of the matrix $M$, that is $\diag M = (M_{11},M_{22},\ldots, M_{nn})^\top$.

Recall that $S = (X^\top X)^{-1}X^\top $ is a $p\times n$ matrix. We have that $\hbeta = Sy  = S\ep +\beta$. Since $\Cov{\ep} = \Sigma$, we have that
$\Cov{\hbeta} = S\Sigma S^\top.$ 
Thus, our goal is to find unbiased estimates of the diagonal of this matrix. The following key lemma re-expresses that diagonal in terms of Hadamard products:

\begin{lemma} 
Let $v$ be a zero-mean random vector, and $M$ be a fixed matrix. Then,
$$\E (M \odot M)(v \odot v)= \diag[M \diag\Cov{v}M^\top].$$
In particular, let $\Sigma$ be a diagonal matrix. and let $\vec\Sigma$ be the vector of diagonal entries of $\Sigma$. Then 
$$ (M \odot M)\vec\Sigma= \diag[M  \Sigma M^\top].$$
Alternatively, let $u$ be a vector. Then 
$(M \odot M)u = \diag[M  \diag(u) M^\top].$

\end{lemma} 

\begin{proof} 
Suppose $M$ has $k$ rows, and
denote them by $m_i$, $i\in[k]$. Let also $\Sigma = \diag\Cov{v}$. 
Then,
for any $i\in[k]$,
the $i$-th entry of the left hand side equals, with $l$ denoting the number of columns of $M$,
$$\E (m_i \odot m_i)^\top(v \odot v)= \E \sum_{j\in[l]} m_{ij}^2 v_{j\in[l]}^2 = \sum_{j\in[l]} m_{ij}^2 \Sigma_j.$$
 The $i$-th entry of the right hand side equals
$m_i^\top\Sigma m_i =\sum_{j\in[l]} m_{ij}^2 \Sigma_{j\in[l]}.$
Thus, the two sides are equal, which proves the first claim of the lemma. 
The second claim follows directly from the first claim, from the special case when the covariance of $v$ is diagonal.  The third claim is simply a restatement of the second one. 

\end{proof} 
We now apply the lemma as follows.
\benum
\item
Let us use the lemma for $v = \ep$ and $M = S$. Notice that we have $\Cov{v} = \Sigma$ is diagonal, so the right hand side  of the lemma is $\diag{S \Sigma S^\top} =  \diag \Cov{\hbeta}$, where the equality follows from our calculation before the lemma. Moreover, the left hand side is $\E (S\odot S) (\ep \odot \ep)  = (S\odot S) \vec\Sigma$, where we vectorize $\Sigma$, writing $\vec\Sigma = (\Sigma_{11}, \ldots, \Sigma_{nn})^\top$. The equality follows because $\Cov{\ep}=\Sigma$ is diagonal. Thus, by the lemma, we have 
$V = \diag \Cov{\hbeta} = (S\odot S) \vec\Sigma.$
\item
Let us now use the lemma for a second time, with $M = I$ and $v = \hep$. This shows that  $\E(\hep\odot\hep) = \diag \Cov{\hep}$. By linearity of expectation, we obtain 
$\E \hV = A \cdot \E(\hep\odot\hep) = A  \cdot \diag \Cov{\hep}.$
\item
Finally, let us use the lemma for the third time, with $M = Q$ and $v = \ep$. As in the first case, the left hand side equals $\E(\hep\odot\hep) = (Q\odot Q) \vec\Sigma$. The right hand side equals $\diag[M \diag\Cov{v}M^\top] = \diag{Q \Sigma Q}$, where we used that $Q$ is a symmetric matrix.  Now, $\Cov{\hep} = \Cov{Q\ep} =Q \Sigma Q$. Thus,  we conclude
$\diag\Cov{\hep} = \diag Q \Sigma Q= (Q \odot Q)\vec\Sigma$.
\eenum

Putting these conclusions together, we obtain that $\hV$ is unbiased, namely  $\E \hV = \diag \Cov{\hbeta}$, if
$A (Q \odot Q)\vec\Sigma= (S\odot S) \vec\Sigma.$
This system of linear equations holds for any $\Sigma$ if and only if
$A (Q \odot Q)= (S\odot S).$
If $Q\odot Q$ is invertible, then we have \eqref{ad}.
This shows that the original estimator $\hV$ has the required form, finishing the proof. 

For a sequence  $(w_p)_{p\ge 1}$, 
we can also estimate $w_p^\top S \Sigma S^\top w_p$,
the variance of $w_p^\top \hbeta$.  
Since we can use $(Q \odot Q)^{-1} (\hep \odot \hep)$
to estimate $\diag(\Sigma)$ in an unbiased way, and considering that $\Sigma$ is diagonal, an unbiased estimator of $w_p^\top S \Sigma S^\top w_p$ 
is $w_p^\top S \diag[(Q\odot Q)^{-1}(\hep \odot \hep)]S^\top w_p$.

\subsection{Proof of Proposition \ref{prop:known}}
\label{pflb}

To prove the lower bound, we first claim that for any symmetric matrix $A$,  
$$\rank A \odot A  \le \binom{\rank A+1}{2}.$$
Therefore, in order for $Q\odot Q$ to be invertible, we need 
$n  \le \binom{n-p+1}{2}.$
By solving the quadratic inequality, this is equivalent to $p \le [2n+1 - (8n+1)^{1/2}]/2$.

To prove the claim about ranks, let
$r$ be the rank of $A$, and let 
$A = \sum_{i=1}^r v_i v_i^\top$ be its eigendecomposition. Here $(v_i)_{i\in [r]}$ are orthogonal, but not necessarily of unit norm. Then, 
$$A \odot A = \left( \sum_{i=1}^r v_i v_i^\top\right) \odot  \left( \sum_{i=1}^r v_i v_i^\top\right)  
=  \sum_{i=1}^r (v_i\odot  v_i) (v_i\odot  v_i)^\top+ 2 \sum_{1 \le i<j \le r}^r (v_i \odot  v_j) (v_i \odot  v_j) ^\top. $$
This shows that the rank of $A \odot A$ is at most $r + r(r-1)/2 = r(r+1)/2$, as desired.

\subsection{Proof of Theorem \ref{thm:Tnorm}}
\label{pf:dd}

Our first step is to reduce to the case $\Gamma = I_p$. Indeed, we notice that we can write $X = Z \Gamma^{1/2}$, where $Z$ is the matrix with rows $z_i$. 
Hence, 
$Q =I_n- X (X^\top X)^{-1} X^\top = I_n-Z (Z^\top Z)^{-1} Z^\top.$
Therefore, we can take $\Gamma = I_p$. 

The next step is to reduce the bounds on eigenvalues to bounds on certain quadratic forms. 
For all $i\in[p]$,
let us define the 
$n\times n$
matrices $R_i = X^\top X - x_ix_i^\top = \sum_{j\neq i} x_j x_j^\top$. See Section \ref{pf:equiv} for a proof of the following result. 

\begin{lemma}[Reduction to quadratic forms]\label{qf} We have the following two bounds on the eigenvalues of $T=Q\odot Q$ 
with  $Q = I_n- X(X^\top X)^{-1}X^\top $:
$$\lambda_{\max}(T) \le \max_{i\in[p]}\frac{1}{1+x_i^\top R_i ^{-1}x_i},$$
and
$$\lambda_{\min}(T) \ge \min_{i\in[p]}\frac{1-x_i^\top R_i ^{-1}x_i}{(1+x_i^\top R_i ^{-1}x_i)^2}.$$
\label{equiv}
\end{lemma}

To bound these expressions, we will use the following well-known statement about concentration of quadratic forms; 
its short proof is provided in Section \ref{pf:quad_form}.  

\begin{lemma}[Concentration of quadratic forms, consequence of Lemma B.26 in \citet{bai2010spectral}] Let $x \in \RR^p$ be a random vector with i.i.d. entries and $\EE{x} = 0$, for which $\EE{(\sqrt{n}x_i)^2} = \sigma^2$ and $\sup_i \EE{(\sqrt{n}x_i)^{4+\eta}}$ $ < C$ for some $\eta>0$ and $C <\infty$. 
Moreover, let $A_p$ be a random $p \times p$ symmetric matrix independent of $x$, with uniformly bounded eigenvalues. Then 

$$\P(|x^\top A_p x - n^{-1} \sigma^2 \tr A_p|^{2+\eta/2}>t) \le C t^{-1} n^{-(1+\eta/4)}.$$
\label{quad_form}
\end{lemma}

Proceeding with the proof of Theorem \ref{thm:Tnorm},
we can scale $X$ so that the variances of the entries of $X$ are $1/n$.
Define the following events:
\begin{equation*}\begin{aligned}
    &\Xi_1 =\bigcap_{i=1}^n \left\{ \left| \frac{1}{p} \tr R_i^{-1}- \frac{\gamma_{p,n}}{1-\gamma_{p,n}} \right|\leq n^{-0.9999} \right\}, 
    \\ &\Xi_2=\bigcap_{i=1}^n \left\{\left|x_i^\top R_i^{-1}x_i - n^{-1}\tr R_i^{-1}\right|<\xi\right\}.
\end{aligned}\end{equation*}

We obtain $P(\Xi_2^c) \le \xi^{-4-\delta/2} n^{-1-\delta/4} $ from Lemma \ref{quad_form} and by taking a union bound over $[n]$. 
Next, we will verify that 
$\P(\Xi_1^c)\le  5n^{-1-\delta/4}$.  Then by Lemma \ref{equiv}, and applying the inequality $\P(A\cap B) \ge 1-\P(A^c)-\P(B^c)$ with $A = \Xi_1$, $B=\Xi_2$, the proof will be concluded.  

Now to bound $\P(\Xi_1^c)$, according to the rank inequality, see Theorem A.43 of \cite{bai2010spectral}, 
we can equivalently show 
$\P\left(\left|p^{-1}\tr R^{-1} - \frac{\gamma_{p,n}}{1-\gamma_{p,n}}\right|\le n^{-0.9999}\right)>1-5n^{-1-\delta/4}$ for sufficiently large $n$.
We further define 
 \begin{equation*}\begin{aligned}
 &\Omega_1:=   \bigg\{\max_{i\in[n],\,j\in[p]}|x_{ij}| \leq n^{-0.001}\bigg\},\\
  &\Omega_2:= \bigcap_{i=1}^n \{(1-\sqrt{\gamma_{p,n}})^2/2 
    \le \lambda_{\min}(R)\le \lambda_{\max}(R) \le 2(1+\sqrt{\gamma_{p,n}})^2\}.
 \end{aligned} \end{equation*}
   Since $\E |X_{ij}|^{8+\delta} < \infty$ for $i\in[n],\,j\in[p]$, it can be checked that 
   $\P(\Omega_1^c) \le n^{-1-\delta/4}$ 
   by Chebyshev's inequality and taking a union bound. To bound $\P(\Omega_2^c)$,
   by the argument from Section 9.12.5 of \cite{bai2010spectral}, it can be readily checked that $\P(\Omega_2^c|\Omega_1)< n^{-\ell}$ for any fixed $\ell>0$. 
   Then 
   by taking $\ell$ large enough,
   $\P(\Omega_2^c) \le \P(\Omega_2^c | \Omega_1) + \P(\Omega_1^c) \le 2 \P(\Omega_1^c)$.

Let $v_0$ be a small positive constant, and take $x_\ell = (1-\sqrt{\gamma_{p,n}})^2/3, x_r = 3(1+\sqrt{\gamma_{p,n}})^2$. 
Denoting the imaginary unit by $\mathrm{i}$,
define the rectangular region 
$$\Upsilon =\{z = x_\ell + \mathrm{i}v: v \in [-v_0,v_0]\} \cup  \{z = x_r + \mathrm{i}v: v \in [-v_0,v_0]\} \cup \{z = x \pm  \mathrm{i}v_0: x \in [x_\ell,x_r]\}$$
in the complex plane.

For any $z$ in $\Upsilon$, let $m_n(z):= p^{-1}\tr (R- z I)^{-1}$, and $m_c(z)$ be the solution to the self-consistent equation $[m_c(z)]^{-1}= -z + (p/n)^{-1} [1+m_c(z)]^{-1}$ whose imaginary part has the same sign as that of $z$. 
This equation has a unique solution \citep{bai2010spectral}.

By Cauchy's integral formula, we have 
    \begin{equation}\label{trlimdif}\left(p^{-1}\tr R^{-1} - \frac{\gamma_{p,n}}{1-\gamma_{p,n}}\right) I(\Omega_1\cap \Omega_2) = \frac1{2\pi \mathrm{i}}\oint_{\Upsilon} \frac{m_n(z) - m_c(z)}{z} I(\Omega_1\cap \Omega_2) dz.
    \end{equation}
Defining $\Omega_3 = \{|m_n(z) - m_c(z)| \le n^{-0.9999}\}$, we have 
$\P(\Omega_3^c|\Omega_1)<n^{-\ell}$ for any fixed $\ell>0$. 
This is a consequence of the averaged local law from random matrix theory, see Theorem 3.5 of \cite{yang2020linear} for instance.
Specifically, we apply their result for 
$d_N = \gamma_{p,n}$ and $\Sigma = I$. 
Note that
for $\rho_{2c}$ defined in their equation (2.10),
$\text{supp}(\rho_{2c}) = [(1-\sqrt{\gamma_{p,n}})^2, (1+ \sqrt{\gamma_{p,n}})^2]$ thus 
with $D_{out}$ from their equation (3.10) and some small $\omega>0$,
$\Upsilon \subset D_{out}$. 
Additionally, the bounded support condition in their equation (3.15) is satisfied on $\Omega_1$. Then (3.21) in their Theorem 3.5 implies our desired bound.

 By taking $\ell$ large enough,
$\P(\Omega_3^c) \le \P(\Omega_3^c|\Omega_1)+\P(\Omega_1^c) \le 2 \P(\Omega_1^c)$.
Therefore, combining the above bounds 
for the probabilities of $\Omega_1,\Omega_2,\Omega_3$ and by \eqref{trlimdif}, we have
\begin{equation*}
\begin{aligned}\P\left(\left|p^{-1}\tr R^{-1} - \frac{\gamma_{p,n}}{1-\gamma_{p,n}}\right| \le  n^{-0.9999}\right) &\ge \P(\Omega_1 \cap \Omega_2 \cap \Omega_3)\\
& \ge 1- \P(\Omega_1^c)-\P( \Omega_2^c)-\P(\Omega_3^c) 
\ge  1-5 \P(\Omega_1^c).
\end{aligned}
\end{equation*}
 This finishes the argument.

\subsection{Proof of Lemma \ref{equiv}}
\label{pf:equiv}

We need to bound the smallest and largest eigenvalues of $T=Q\odot Q$. 
Now, for all $i,j\in[n]$, 
$T_{ij} = Q_{ij}^2 = (\delta_{ij} - x_i^\top R^{-1} x_j)^2$, where $R = X^\top X$ and $\delta_{ij}$ is the Kronecker delta which equals unity if $i=j$, and zero otherwise. 
We will use the following well-known rank-one perturbation formula for an invertible matrix $T$ and a vector $u$ of conformable size: 
$$(uu^\top +T)^{-1} = T^{-1} - \frac{T^{-1} u u^\top T^{-1}}{1+ u^\top T^{-1}u}.$$

We will also use a ``leave-one-out'' argument which has roots in random matrix theory \citep[see e.g.,][]{bai2010spectral, paul2014random, yao2015large}. 
For any $i\in[n]$,
letting 
$R_i = X^\top X - x_ix_i^\top = \sum_{j\neq i} x_j x_j^\top$,
we have $R^{-1} = R_i ^{-1} - \frac{R_i ^{-1}x_ix_i^\top R_i ^{-1}}{1+x_i^\top R_i ^{-1}x_i}.$
For all $i,j\in[n]$, 
the quantity that is squared in the $i,j$-th entry of $T$ is
 thus
$$x_i^\top R^{-1} x_j 
=x_i^\top R_i ^{-1} x_j -  \frac{x_i^\top  R_i ^{-1}x_i \cdot x_i^\top R_i ^{-1} x_j}{1+x_i^\top R_i ^{-1}x_i}
=\frac{x_i^\top R_i ^{-1} x_j }{1+x_i^\top R_i ^{-1}x_i}.$$
Also, for all $i\in[n]$,  we have
$x_i^\top R^{-1} x_i
= \frac{x_i^\top R_i ^{-1} x_i}{1+x_i^\top R_i ^{-1}x_i},
 $
so that the diagonal terms are 
$$T_{ii} =  (1-x_i^\top R^{-1} x_i)^2 =\frac{1}{(1+x_i^\top R_i ^{-1}x_i)^2}.$$

By the Gershgorin disk theorem \citep[][Thm 6.1.1]{horn1990matrix}, we have 
$\lambda_{\max}(T) \le $ $ \max_{i\in[n]}$ $\left(T_{ii} + \sum_{j\neq i} |T_{ij}| \right).$
Thus,
$$\lambda_{\max}(T)  \le \max_{i\in[n] }\frac{1+\sum_{j\neq i} (x_i^\top R_i ^{-1} x_j)^2}{(1+x_i^\top R_i ^{-1}x_i)^2}.$$

Now, the sum in the numerator can be written as $x_i^\top R_i ^{-1} (\sum_{j\neq i} x_j x_j^\top) R_i ^{-1} x_i = x_i^\top R_i ^{-1} x_i$. Thus, the upper bound simplifies to
$\max_{i\in[n]}1/(1+x_i^\top R_i ^{-1}x_i).$
Similarly, for the smallest eigenvalue, by the Gershgorin disk theorem \citep[][Thm 6.1.1]{horn1990matrix}, we have 
$\lambda_{\min}(T) \ge \min_{i\in[n]}$ $\left(T_{ii} - \sum_{j\neq i} |T_{ij}| \right).$
We can express for all $i\in[n]$
$$T_{ii} - \sum_{j\neq i} |T_{ij}| = \frac{1-x_i^\top R_i ^{-1}x_i}{(1+x_i^\top R_i ^{-1}x_i)^2}.$$
Hence
$\lambda_{\min}(T) \ge \min_{i\in[n]}(1-x_i^\top R_i ^{-1}x_i)(1+x_i^\top R_i ^{-1}x_i)^2,$
finishing the proof. 

\subsection{Proof of Lemma \ref{quad_form}}
\label{pf:quad_form}
We will use the following Lemma quoted from \cite{bai2010spectral}. 

\begin{lemma}[Trace Lemma, Lemma B.26 of \cite{bai2010spectral}]\label{a4}
Let $y$ be a p-dimensional random vector of i.i.d. elements with mean zero. 
Suppose that $\EE{y_i^2} = 1$ for all $i\in[p]$, 
and let $A_p$ be a fixed $p \times p$ matrix. 
Then for any $q\ge 2$,
$$\EE{|y^\top A_p y-\tr A_p|^q} \le C_q \left\{\left(\EE{y_1^4}\tr[A_pA_p^\top]\right)^{q/2}+\EE{y_1^{2q}}\tr[(A_pA_p^\top)^{q/2}]\right\},$$ 
for some constant $C_q$ that only depends on $q$. 
\label{trace_lemma}
\end{lemma}

\begin{proof}
Under the conditions of Lemma \ref{quad_form}, the operator norms $\|A_p\|_2$ are bounded by a constant $C$, thus $\tr[(A_pA_p^\top)^{q/2}] \le p C^q$ and $\tr[A_pA_p^\top] \le p C^2$. 
Consider now a random vector $x$ with the properties assumed in the present lemma. For $y = \sqrt{n}x/\sigma$ and $q = 2+\eta/2$ with $\eta>0$, using that $\EE{y_i^{2q}}\le C$ and the other the conditions in Lemma \ref{quad_form}, Lemma \ref{trace_lemma} thus yields
$$\frac{n^q}{\sigma^{2q}}\EE{|x^\top A_p x-\frac{\sigma^2}{n}\tr A_p|^q} \le C \left\{\left( p C^2\right)^{q/2}+pC^q\right\},$$ 
or equivalently $\EE{|x^\top A_p x-\frac{\sigma^2}{n}\tr A_p|^{2+\eta/2}} \le C n^{-(1+\eta/4)}$. 
By Markov's inequality applied to the $2+\eta$-th moment of $\varepsilon_p = x^\top A_p x-\frac{\sigma^2}{n}\tr A_p$, we obtain as required that for $t>0$,
$\P(|\ep_p|^{2+\eta/2}>t) \le C t^{-1} n^{-(1+\eta/4)}.$
\end{proof}

\subsection{Proof of Proposition \ref{df}}
\label{dfpdf}

We need to evaluate $\E \hV^{\odot 2} = \E \hV \odot \hV \in \R^p$.  Note that this vector is the diagonal of $\E \hV \hV^\top$, which is equal to
\[
\begin{aligned}
\E \hV \hV^\top &= \E A (\hep \odot \hep ) (\hep \odot \hep)^\top A^\top
= \E A \left[(\hep \hep^\top ) \odot  (\hep \hep^\top) \right] A^\top = A \E \left[(\hep \hep^\top ) \odot  (\hep \hep^\top) \right] A^\top.
\end{aligned}
\]
Note that $\hep \hep^\top = Q \ep \ep^\top Q$,
since the residuals $\hep = Q \ep$. Using this expression and recognizing that $\ep$ has \iid $\N(0, \sigma^2)$ entries, 
for any $i,j\in[n]$,
the $(i,j)$-th element of $\E (\hep \hep^\top)^{\odot 2}$ is
\[
\begin{aligned}
&\E \left( \sum_{1 \le l, k \le n} Q_{il} \ep_l \ep_k Q_{kj} \right)^2\\
& = \sum_{l \ne k} \E\left(  Q_{il}^2 Q_{jk}^2 \ep_l^2\ep_k^2 + Q_{il} Q_{jk}Q_{ik} Q_{jl}\ep_k^2\ep_l^2 + Q_{il} Q_{jl} Q_{ik} Q_{jk} \ep_l^2\ep_k^2 \right)  + \sum_{l=1}^n \E Q_{il}^2Q_{jl}^2 \ep_l^4\\
& = \sum_{l \ne k} \left(  Q_{il}^2 Q_{jk}^2 \sigma^4 + Q_{il} Q_{jk}Q_{ik} Q_{jl} \sigma^4 + Q_{il} Q_{jl} Q_{ik} Q_{jk} \sigma^4 \right)  + \sum_{l=1}^n Q_{il}^2Q_{jl}^2 3\sigma^4\\
& = \sigma^4\sum_{l \ne k} \left(  Q_{il}^2 Q_{jk}^2 + 2Q_{il} Q_{jk}Q_{ik} Q_{jl} \right)  + 3\sigma^4\sum_{l=1}^n Q_{il}^2Q_{jl}^2\\
& = \sigma^4\sum_{1 \le l, k \le n} \left(  Q_{il}^2 Q_{jk}^2 + 2Q_{il} Q_{jk}Q_{ik} Q_{jl} \right)
= \sigma^4\sum_{1 \le l, k \le n} Q_{il}^2 Q_{jk}^2 + 2\sigma^4\left( \sum_{l=1}^n Q_{il}Q_{jl}\right)^2.
\end{aligned}
\]
To proceed, we recognize that $\sum_{1 \le l, k \le n} Q_{il}^2 Q_{jk}^2$ is the $(i,j)$-th element of
\[
\left[ (Q\odot Q ) 1_n\right] \left[ (Q\odot Q ) 1_n\right]^\top = (Q\odot Q ) 1_n 1_n^\top (Q\odot Q ),
\]
and $\left( \sum_{l=1}^n Q_{il}Q_{jl}\right)^2$ is the $(i,j)$-th element of
$Q^2 \odot Q^2 = Q \odot Q$.

Summarizing the calculation above, we obtain
\[
\E (\hep \hep^\top)^{\odot 2} =  \sigma^4 (Q\odot Q ) 1_n 1_n^\top (Q\odot Q ) + 2\sigma^4 Q \odot Q,
\]
from which it follows that
\[
\begin{aligned}
\E \hV \odot \hV &= \diag\left[A \left( \sigma^4 (Q\odot Q ) 1_n 1_n^\top (Q\odot Q ) + 2\sigma^4 Q \odot Q \right) A^\top \right]\\
                       & = \sigma^4\diag\left[ A (Q\odot Q ) 1_n 1_n^\top (Q\odot Q ) A^\top \right]+ 2\sigma^4 \diag\left[A (Q \odot Q) A^\top \right]\\
                       & = \sigma^4\diag\left[ (S\odot S ) 1_n 1_n^\top (S\odot S)^\top \right]+ 2\sigma^4 \diag\left[(S \odot S) (Q \odot Q)^{-1} (S \odot S)^\top \right].
\end{aligned}
\]
Note that $V = \sigma^2 \diag\left[(X^\top X)^{-1} \right]$ due to the assumption of homoskedasticity.
Recalling \eqref{E}, 
we find \eqref{eq:d_formula}, finishing the proof.

\subsection{Calculation for the case when \texorpdfstring{$p = 1$}{p = 1}}
\label{p1pf}
We compute each part of the unbiased estimator in turn. We start by noticing that $S = (X^\top X)^{-1} X^\top = X^\top$ is a $1\times n$ vector. We continue by calculating $Q \odot Q$, where $Q =  I- X(X^\top X)^{-1}X^\top   = I- XX^\top$. Thus, 
\begin{align*}
Q_{ij}^2 &=
\begin{cases}
X_i^2 X_j^2 ,& i\neq j\\[2ex]
(1-X_i^2)^2, &\text{else.}
\end{cases}
\end{align*}
Denoting $u = X\odot X$, and $D = I - 2 \diag(X\odot X)$, we can write 
$Q\odot Q = D + uu^\top.$
Now, the estimator takes the form  $\hV = (S\odot S) (Q\odot Q)^{-1}(\hep\odot\hep)$. Hence, we need to calculate $(S\odot S) (Q\odot Q)^{-1} = (X\odot X) (D + uu^\top)^{-1}$. We use the rank-one perturbation formula
$u^\top( D + uu^\top)^{-1} = \frac{ u^\top D^{-1}}{u^\top D^{-1}u+ 1}.$
In our case, 
$$u^\top D^{-1}u = \sum_{j=1}^n \frac{u_j^2}{D_j}= \sum_{j=1}^n  \frac{ X_j^4}{1-2X_j^2},$$
and 
$u^\top D^{-1}$ has entries $X_j^2/(1-2X_j^2)$ for $j\in[n]$. 
This leads to the desired final answer:
$$\hV = u^\top( D + uu^\top)^{-1}\hep\odot\hep = \frac{ \sum_{j=1}^n  \frac{X_j^2}{1-2X_j^2} \hep_j^2}{1+\sum_{j=1}^n  \frac{ X_j^4}{1-2X_j^2}}.$$
Next,
since $X^\top X = 1$, we have for $E$ from \eqref{E} that $E=1$. Finally, for $d$ from \eqref{eq:d_formula},
since $S=X^\top$, $u = X\odot X$, and $Q\odot Q = D + uu^\top$, so that $u^\top 1_n = 1$, we find
$$d = \frac{1}{u^\top (D + uu^\top)^{-1} u} = 1+\frac{1}{u^\top D^{-1}u}=1+\frac{1}{ \sum_{j=1}^n  \frac{ X_j^4}{1-2X_j^2}},
$$
as desired.

\subsection{Proof of Proposition \ref{bc}}
\label{pf:bc}

To compute the bias of White's estimator defined in \eqref{w}, we proceed as follows. First we need to compute its expectation, 
$$\E\hC_{\textnormal{W}} = (X^\top X)^{-1}[X^\top \E \diag(\hep \odot \hep) X] (X^\top X)^{-1}.$$
As we saw,
$\E (\hep \odot \hep)=\diag\Cov{\hep} = \diag Q \Sigma Q= (Q \odot Q)\vec\Sigma.$
Thus, 
$$\diag\E\hC_{\textnormal{W}} = \diag[S  \diag [(Q \odot Q)\vec\Sigma] S^\top] = (S\odot S)  (Q \odot Q) \vec\Sigma .$$

Again, as we saw, $V = \diag \Cov{\hbeta} = (S\odot S)\vec\Sigma.$
Therefore, the bias of White's estimator is as in \eqref{bw}.

To compute the bias of MacKinnon-White's estimator, we proceed similarly, starting with its expectation:  
$$\E\hC_{\mathrm{MW}} = (X^\top X)^{-1}[X^\top \E \diag(Q)^{-1}\diag(\hep \odot \hep) X] (X^\top X)^{-1}.$$
In this equation, the expression $\diag(Q)$ is interpreted as the diagonal matrix  whose entries are those on the diagonal of $Q$. Thus, 
$$\diag\E\hC_{\mathrm{MW}} = \diag[S  \diag(Q)^{-1}\diag [(Q \odot Q)\vec\Sigma] S^\top] = (S\odot S)  (Q \odot Q) \diag(Q)^{-1}\vec\Sigma. $$
Thus the bias is as in \eqref{bmw}, finishing the proof.

\subsection{Proof of Theorem \ref{r}}
\label{pf:r}

We aim to bound $\|\hV - V\|$, where $\|\cdot\|$ denotes usual Euclidean vector norm. 
Recalling that 
$V =  (S\odot S) \vec\Sigma$
and
$\hV =  (S\odot S) (Q\odot Q)^{-1}(\hep\odot\hep)$,
where $S = (X^\top X)^{-1} X^\top$,
we have
\begin{equation}\label{hVVdif}\|\hV - V\| \le \| S\odot S\|_{\op} \|(Q\odot Q)^{-1}\|_{\op} \|(\hep\odot\hep)- (Q\odot Q) \vec\Sigma\|.
\end{equation}

We will find upper bounds for each term in the above product.

\benum

\item To bound $\| S\odot S\|_{\op}$,
Schur's inequality \citep[e.g.,][Thm. 5.5.1]{horn1991topics}, states that
$\| S\odot S\|_{\op} \le \|S\|_{\op}^2.$
Moreover, as $\|S\|_{\op} = 1/\sigma_{\min}(X)$, 
it follows by  (9.7.9) in \cite{bai2010spectral} that
for any constant $c>0$ and $\ell>0$, 
$\sigma_{\min}(X) \ge \sigma_{\min}(\Gamma^{1/2}) (n^{1/2} - p^{1/2} - c)$ holds with probability $1- o(n^{-\ell})$.
Thus, denoting $$\mathcal{E}_{1,n} :=\left\{ n\|S\odot S\|\le c\frac{1}{\sigma_{\min}(\Gamma)(1 - \gamma_{p,n}^{1/2})^2}\right\},$$
we have $\P (\mathcal{E}_{1,n}) \ge 1- o(n^{-\ell})$
for any constant $c>1$ and $\ell>0$.

\item To bound  $ \|(Q\odot Q)^{-1}\|_{\op}$,
denoting $$\mathcal{E}_{2,n}:= \left\{\|(Q \odot Q)^{-1}\|_{\op} \le c \frac{1}{(1 - \gamma_{p,n})(1-2\gamma_{p,n})} \right\},$$
by Theorem \ref{thm:Tnorm}, we have
$\P(\mathcal{E}_{2,n})\ge 1-O(n^{-1-\delta/4}) $ for any $c>1$.


\item To bound  $\alpha:= \|(\hep\odot\hep)- (Q\odot Q) \vec\Sigma\|$,
we can express $\alpha^2 = \sum_{i=1}^n \alpha_i^2$, where 
for all $i\in[n]$,
$\alpha_i^2 = (\hep_i^2 - (q_i \odot q_i)^\top \vec\Sigma)^2.$
From the earlier unbiasedness argument,
$\E\hep_i^2 = (q_i \odot q_i)^\top \vec\Sigma$, and thus $\E \alpha_i^2 = \Var{\hep_i^2}$.
A simple calculation shows that, 
with $\Gamma_k = \E\ep_k^4 $
 for all $k\in[n]$, we have
$$ \Var{\hep_i^2} = \sum_{k=1}^n q_{ik}^4 (\Gamma_k - 3 \Sigma_k^2) + 2 [ (q_i \odot q_i)^\top \vec\Sigma]^2.$$

Now  for all $k\in[n]$,
the excess kurtosis can be bounded as $\Gamma_k - 3 \Sigma_k^2 \le (C-3) \Sigma_k^2$. Therefore, we can bound by Markov's inequality:
$$\mathbb{P}(\alpha \ge t) \le \frac{\sum_{i=1}^n \E\alpha_i^2}{t^2} 
= \frac{(C-1)\sum_{i=1}^n  [ (q_i \odot q_i)^\top \vec\Sigma]^2}{t^2} 
\le \frac{(C-1)\cdot \left\| (Q\odot Q) \vec\Sigma\right\|^2}{t^2}.$$
According to  Theorem \ref{thm:Tnorm},
$\|Q\odot Q\|_{\op} \le c (1-\gamma_{p,n})$ 
holds with probability $1-C' n^{-1-\delta/4}$ for some positive constant $C'$ and any constant $c>1$. Hence
$$ \P(\alpha \ge t) \le \frac{2c(1-\gamma_{p,n})^2 \|\vec\Sigma\|^2}{t^2} + C' n^{-1-\delta/4}.$$

\eenum
In conclusion, we have for sufficiently large $n$ and $p$ with $\gamma_{p,n}<1/2$ that 
\begin{equation*}
    \begin{aligned}
        &\P \left( \frac{\|\hV - V\|}{\|\vec\Sigma\|}>\frac{t}{n}\right) \le 
        \P \left( \alpha n\|S\odot S\|_{\op} \| (Q\odot Q)^{-1}\|_{\op}> t\|\vec\Sigma\|\right)\\
        & \le \mathbb{P} \left( \alpha n\|S\odot S\|_{\op} \| (Q\odot Q)^{-1}\|_{\op}> t\|\vec\Sigma\| \bigg| \mathcal{E}_{1,n}, \mathcal{E}_{2,n}\right)+ \P(\mathcal{E}_{1,n}^c)+ \P(\mathcal{E}_{2,n}^c)\\
        & \le \frac{2c}{t^2}\frac{1}{\left[\sigma_{\min}(\Gamma)(1-\gamma_{p,n}^{1/2})^2 (1-2 \gamma_{p,n})\right]^2} + C' n^{-1-\delta/4}.
    \end{aligned}
\end{equation*}
This proves the required result.

\subsection{Proof of Theorem \ref{thmhatviasym}}
\label{sec:pfhatvnormal}

For the given $Z$, 
any matrices $M_1,M_2\in \mathbb{R}^{n\times n}$ and any vector $v\in \mathbb{R}^p$, we find
\begin{equation*}
\begin{aligned}
&v^\top \diag\left\{M_1 [(M_2 Z)\odot(M_2 Z)]\right\} v = \sum_{i=1}^n v_i^2 \sum_{j=1}^n M_{1,ij}\left(\sum_{k=1}^n M_{2,jk}Z_k\right)^2\\
&= \sum_{k_1,k_2=1}^n Z_{k_1}Z_{k_2} M_{2,jk_1}\left(\sum_{i,j=1}^n v_i^2 M_{1,ij}\right)M_{2,jk_2} 
= Z^\top M_2^\top \diag[(v\odot v)^\top M_1]M_2 Z.
\end{aligned}
\end{equation*}
Taking  $v = Sw_p$, $M_1 = (Q\odot Q)^{-1}$, $M_2 = Q\Sigma^{1/2}$ above, 
we find that \begin{equation}\label{quadformZ}w_p^\top S (\diag\widehat{\vec\Sigma}) S^\top w_p=Z^\top G(w_p) Z,\end{equation} 
where $G(w_p)$ is defined in \eqref{gw}.

Recall the following claim from e.g., \cite{bai2010spectral}. 
Let $X=\left(X_1, \ldots, X_n\right)$, where $X_i$, $i\in[n]$, are i.i.d.~real random variables with mean zero and variance one. Let $B=\left(b_{i j}\right)_{i,n\in[n]}$ be a real symmetric matrix. Then we have
\begin{equation}\label{Varform}
\Var (X^\top B X) =\left(\E\left|X_1\right|^4-3\right) \sum_{i=1}^n b_{i i}^2 + 2 \tr B^2.
\end{equation}
Let 
$g:\R^n\to \R$ be defined for all $z\in\R^n$ by
$g(z) = w_p^\top S (\diag\widehat{\vec\Sigma}) S^\top w_p = z^\top G(w_p) z$. 
By the moment assumptions on the entries of $Z$, we have \begin{equation}\label{VhatV}
\min\{C_4-1,2\}\|G(w_p)\|_{\Fr}^2 \le \Var{g(Z)} \le \max\{C_4-1,2\}\|G(w_p)\|_{\Fr}^2 \le (C_4+1)\|G(w_p)\|_{\Fr}^2.\end{equation}

To use the second order Poincar\'e inequality, see \cite{chatterjee2009fluctuations}, Theorem 2.2, we need to bound the following quantities:
$$\kappa_0 = \left(\E \sum_{j=1}^n \left|\frac{\partial g}{\partial z_j}(Z)\right|^4\right)^{1/2},\quad 
\kappa_1 = [\E \|\nabla_z g(Z)\|^4]^{1/4}, \quad 
\kappa_2 = [\E \|\nabla^2_z g(Z)\|_{\op}^4]^{1/4}.$$

In the following,
we denote for all $j\in[n]$ by $G_j^\top$ the $j$-th row of $G=G(w_p)$, and omit the dependence on $w_p$ for simplicity. A direct calculation yields
 \begin{equation*}
    \nabla g = 2 G Z, \quad \frac{\partial g}{\partial z_j} = 2 G_j^\top Z, \quad \nabla^2 g = 2G.
\end{equation*}

We have
\begin{equation*}
    2^{-4}\E \left|\frac{\partial g}{\partial z_j}(Z)\right|^4 = \E(Z^\top G_j G_j Z)^2 \le 2\Var Z^\top G_j G_j^\top Z + 2 C_4 \|G_j^\top G_j\|^2 \leq (4C_4+2) \|G_j\|^4,
\end{equation*}
where we use \eqref{Varform} in the last step.
Then
$\kappa_0  \leq 4 (4C_4+2)^{1/2} 
\left(\sum_{j\in[n]} \|G_j\|^4\right)^{1/2} \leq 4 (4C_4+2)^{1/2} \|G\|_{\Fr}\|G\|$.
We also have $\kappa_2=2\|G\|$ and
\begin{equation*}
    2^{-1}\kappa_1 =(\E (Z^\top G^\top G Z)^2)^{1/4} \leq (4C_4+2)^{1/4} \|G^\top G\|_{\Fr}^{1/2} \leq (4C_4+2)^{1/4}
    \|G\|_{\Fr}.
\end{equation*}
Therefore, we have \begin{equation*}\begin{aligned}
     &d_{\mathrm{TV}}\left(\frac{w_p^\top S (\diag\widehat{\vec\Sigma}) S^\top w_p- w_p^\top S\Sigma S^\top w_p}{\sqrt{\Var{[w_p^\top S (\diag\widehat{\vec\Sigma}) S^\top w_p] }}},  \N(0,1)\right) \\
     &\le \frac{2\sqrt{5}(c_1 c_2 \kappa_0+c_1^3 \kappa_1 \kappa_2)}{\min\{C_4-1,2\}\|G\|_{\Fr}^2} \le \frac{\left[8\sqrt{5}(4C_4+2)^{1/2} c_1 c_2+8\sqrt{5}(4C_4+2)^{1/4} c_1^3 \right] \|G\|}{\min\{C_4-1,2\}\|G\|_{\Fr}}.
\end{aligned}\end{equation*}
This proves the desired claim.

\subsection{Proof of Proposition \ref{lemWnorm}}\label{sec:pflemWnorm}
\begin{proof}

We consider $\|G(w_p)\|_{\Fr}$ first. For any vector $v\in \mathbb{R}^p$,  we find that \begin{equation}\label{Wifr}
\begin{aligned}
   & \|Q \diag(v) Q\|_{\Fr}^2 = \sum_{j,k=1}^n
    \left(\sum_{\ell=1}^n Q_{j\ell} v_{\ell} Q_{\ell k}\right)^2 \\&= \sum_{j,k=1}^n \sum_
    {\ell_1,\ell_2=1}^n Q_{j\ell_1}v_{\ell_1}Q_{\ell_1 k}Q_{j\ell_2}v_{\ell_2}Q_{\ell_2 k} =\sum_{\ell_1,\ell_2=1}^n v_{\ell_1} [(Q^2)_{\ell_1\ell_2}]^2 v_{\ell_2} = v^\top (Q\odot Q) v ,
\end{aligned}\end{equation}
where we use condition \ref{cond3lemWnorm} in the first step, and $Q^2 =Q$ in the second last step.\\

Recall that for $j\in[n]$,
$S_{.j}$ is the $j$-th column of $S$.
 Then
\begin{equation}\label{lbdwShad}\begin{aligned}\|(w_p^\top S) \odot  (w_p^\top S)\| = \left(\sum_{j=1}^n  (w_p^\top S_{.j})^4 \right)^{1/2}  &\geq n^{-1/2}\sum_{j=1}^n (w_p^\top S_{.j})^2 \\& = n^{-1/2}w_p^\top S S^\top w_p   =  n^{-1/2}w_p^\top (X^\top X)^{-1} w_p.
\end{aligned}\end{equation}
Substituting $v^\top =[(w_p^\top S) \odot  (w_p^\top S)] (Q\odot Q)^{-1} $ into \eqref{Wifr}, 
by the above bound, 
and by the conditions $\|Q\odot Q\|\ge c$ and $\lambda_{\min}(\Sigma)>c$, we conclude that  
$$\|G(w_p)\|_{\Fr} =\Omega(n^{-1/2}[\lambda_{\max}(X^\top X)]^{-1}).$$


The upper bound for $\|G(w_p)\|$ is obtained by \begin{equation*}\begin{aligned}
\|G(w_p)\| &\le \|Q\Sigma Q\|\|\diag(\left[(w_p^\top S)\odot (w_p^\top S)\right] (Q \odot Q)^{-1})\| \\ 
& \le C \max_{j\in[n]}\left|\left[(w_p^\top S)\odot (w_p^\top S)\right] (Q \odot Q)^{-1}e_j\right| = o \left(n^{-1/2}[\lambda_{\min}(X^\top X)]^{-1}\right),
\end{aligned}\end{equation*}
where the second step uses condition \ref{cond3lemWnorm} and $\|Q\|\le 1$, and the third step uses conditions \ref{cond1lemWnorm} and \ref{cond2lemWnorm}.

In the remainder, 
we show that the first two conditions hold 
probability tending to one under the random design $X$ satisfying the conditions of Theorem \ref{thm:Tnorm}. 
In a matrix form, write $X = Z \Gamma^{1/2}$ with the $i$-th row being $x_i= \Gamma^{1/2} z_i$ for all $i\in[n]$,
where $z_i$ has independent entries with zero mean, unit variance and finite $(8+\delta)$-th moment. The bound for $Q\odot Q$ in condition \ref{cond1lemWnorm} holds with probability tending to one, as a consequence of Theorem \ref{thm:Tnorm}.

Next we verify condition \ref{cond2lemWnorm}. 
For any sequence of deterministic vectors $w_p\in \mathbb{R}^p$ with a bounded norm, we have 
$$\E|w_p^\top (Z^\top Z)^{-1} z_k|^{8+\delta}
\le  
\E \bigg|w_p^\top \big(\sum_{i\neq k} z_i z_i^\top\big)^{-1} z_k\bigg|^{8+\delta}
\le C \E\biggl\|w_p^\top \bigg(\sum_{i\neq k} z_i z_i^\top\bigg)^{-1}\biggl\|^{8+\delta} 
= O(n^{-8-\delta}),$$
where in the first step we use the Sherman–Morrison formula 
$$w_p^\top (Z^\top Z)^{-1} z_k 
= w_p^\top   \bigg(\sum_{i\neq k} z_i z_i^\top \bigg)^{-1} z_k/ \bigg[1+z_k^\top \bigg (\sum_{i\neq k} z_i z_i^\top \bigg)^{-1} z_k \bigg]
\leq  w_p^\top   \bigg(\sum_{i\neq k} z_i z_i^\top \bigg)^{-1} z_k.$$ and in the second step we use the moment bound $\E|w_p^\top z_k|^q \le C\|w\|^q$ for any $0\le q \le 8+\delta$.
This moment bound can be checked by applying Lemma \ref{a4} with $A_p = w_p w_p^\top$ and using $\tr(w_p w_p^\top) = \|w_p\|^2$.

Therefore \begin{equation*}\begin{aligned} \P &  \left(\max_{k\in[n]}\left|w_p^\top S_{.k}\right|>n^{-7/8}\lambda_{\min}^{-1/2}(\Gamma) \right) 
 =  \P 
\left(\max_{k\in[n]}\left|w_p^\top \Gamma^{-1/2}(Z^\top Z)^{-1}z_k\right|> n^{-7/8}\lambda_{\min}^{-1/2}(\Gamma) \right)\\
&\le n \mathbb{P} \left(\left|w_p^\top \Gamma^{-1/2}(Z^\top Z)^{-1}z_k\right|> n^{-7/8}\lambda_{\min}^{-1/2}(\Gamma)\right) \\
& \le n^{1+\frac{7}{8}(8+\delta)}\lambda_{\min}^{\frac{1}{2}(8+\delta)}(\Gamma)\E |w_p^\top \Gamma^{-1/2}(Z^\top Z)^{-1}z_k|^{8+\delta} = O(n^{-\delta/8}).
\end{aligned}\end{equation*}

We further find $\lambda_{\min}^{-1}(\Gamma) = O_P( n[\lambda_{\min}(X^\top X)]^{-1}),$
 where  we use $\lambda_{\min}(X^\top X) \le \lambda_{\max}(\Gamma) \lambda_{\min}(Z^\top Z) $ $= O_P(n \lambda_{\max}(\Gamma)) $
 and that $\kappa(\Gamma)$ is bounded.
 Then we find 
 $$\max_{k\in[n]}|w_p^\top S_{.k}| = O_P(n^{-3/8}[\lambda_{\min}(X^\top X)]^{-1/2}]).$$
This shows the validity of condition \ref{cond2lemWnorm} with probability tending to one.
\end{proof}

\subsection{Consistency of the SNR estimator}\label{sec:SNR}
The following result shows that the SNR estimator given in \eqref{eq:snr_estimator} is ratio-consistent.
\begin{proposition}[Ratio-consistency of SNR estimator]
Assume the conditions of Theorem \ref{thmhatviasym} on 
the noise $\ep$, 
and conditions \ref{cond1lemWnorm} and \ref{cond3lemWnorm} of 
Proposition \ref{lemWnorm} on the data matrix $X$. 
In addition, suppose $\|\beta\| = \Omega(1)$ and $\lambda_{\min}(n^{-1} X^\top X)\asymp 1$.
For $\widehat{\snr}$ from \eqref{eq:snr_estimator},
and the signal-to-noise ratio $\snr = n \|\beta\|^2/\tr \Sigma$,
we have $\widehat{\snr}/\snr \to_P 1$.
\end{proposition}

\begin{proof}
We already know that the numerator and the denominator of $\widehat{\snr}$ are unbiased estimators for $\|\beta\|^2$ and $n^{-1}\tr \Sigma$, respectively. The conclusion follows 
by  applying Slutsky's theorem
if we can show 
\begin{equation}\label{hbetapart}
    (\|\hbeta\|^2 - 1_p^\top \hV)/\|\beta\|^2 \to_P 1,
\end{equation}
and \begin{equation}\label{hVpart}
   \frac{1_p^\top (Q \odot Q)^{-1} (\hep \odot \hep)}{\tr \Sigma}\to_P 1.
\end{equation}

 We show \eqref{hbetapart} first by verifying that $\mathrm{Var}\left[(\|\hbeta\|^2 - 1_p^\top \hV)/\|\beta\|^2\right] \to 0$.
We have a decomposition for $\|\hbeta\|^2$, given as $\|\hbeta\|^2 = \|\beta + S \ep\|^2 = \|\beta\|^2 +2 \beta^\top S \ep + \ep^\top S^\top S \ep$. 
By the assumptions $\|\beta\| = \Omega(1)$, $\lambda_{\min}(X^\top X) \asymp n$, and $\|\Sigma\|\asymp 1$, we find
\begin{equation}\label{varhbeta}
\begin{aligned}\mathrm{Var}(\|\hbeta\|^2) &\le 8 \mathrm{Var}[\beta^\top S \ep] + 2 \mathrm{Var}[\ep^\top S^\top S \ep] \\
& \le C\beta^\top (X^\top X)^{-1}X^\top \Sigma  X (X^\top X)^{-1} \beta + C \|\Sigma S^\top S\|_{\Fr}^2 = O(\|\beta\|^2 n^{-1}).
\end{aligned}
\end{equation}
For $1_p^\top \hV$, using \eqref{quadformZ} with $w_p= e_i$ and summing over $i \in [n]$, 
we find $1_p^\top \hV = Z^\top \Sigma^{1/2}Q \diag[1_p^\top (S\odot S) (Q\odot Q)^{-1}] Q \Sigma^{1/2} Z$.
Then
\begin{align*}
    \mathrm{Var}(1_p^\top \hV) \asymp \|\Sigma^{1/2}Q \diag[1_p^\top (S\odot S) (Q\odot Q)^{-1}] Q \Sigma^{1/2}\|_{\Fr}^2 \asymp 1_p^\top (S\odot S)(Q\odot Q)^{-1}(S\odot S) 1_p,
\end{align*}
where in the second step we use \eqref{Wifr} with $v^\top = 1_p^\top (S \odot S)(Q\odot Q)^{-1}$.
By $\|S\odot S\| \le \|S\|_{\op}^2 \asymp n^{-1}$ and $(Q\odot Q)^{-1}\asymp 1$, we have $\mathrm{Var}(1_p^\top\hV) = O(n^{-1})$.
Therefore, combining this with \eqref{varhbeta}
and $\|\beta\|\asymp 1$, we conclude \eqref{hbetapart}.

Then we verify \eqref{hVpart} by checking that the variance of the left term tends to zero.
Writing $1_p^\top (Q\odot Q)^{-1}(\hep\odot \hep) = Z^\top \Sigma^{1/2}Q \diag[1_p^\top (Q\odot Q)^{-1}] Q \Sigma^{1/2}Z,$
we obtain 
\begin{equation*}\begin{aligned}
    \mathrm{Var}\left(1_p^\top (Q\odot Q)^{-1}(\hep\odot \hep)\right) &\asymp \|\Sigma^{1/2}Q \diag[1_p^\top (Q\odot Q)^{-1}] Q \Sigma^{1/2}\|_{\Fr}^2 \asymp  1_p^\top (Q\odot Q)^{-1}1_p \asymp n,
\end{aligned}\end{equation*}
where in the second step we use \eqref{Wifr}. Since $\tr \Sigma \asymp n$, we conclude \eqref{hVpart}.
   
\end{proof}

\subsection{Proof of Lemma \ref{lem:cons-est} and Theorem \ref{thmbetanormal}}\label{sec:pfinf}
\subsubsection{Proof of Lemma \ref{lem:cons-est}}

We start with the first conclusion.
By \eqref{VhatV} and the conclusion $\|G(w_p)\|_{\Fr} = o(n^{-1/2}[\lambda_{\min}(X^\top X)]^{-1})$ from Proposition \ref{lemWnorm}, we have  \begin{equation}\label{varestvar}\Var{\{w_p^\top S \diag\left[(Q\odot Q)^{-1}(\hep \odot \hep)\right] S^\top w_p\}}=o (n^{-1}[\lambda_{\min}(X^\top X)]^{-2}).
\end{equation}
 This and $w_p^\top S \Sigma S^\top w_p = w_p^\top (X^\top X)^{-1}X^\top \Sigma X (X^\top X)^{-1}w_p \asymp [\lambda_{\min}(X^\top X)]^{-1}$ yield  $$ \frac{w_p^\top S \diag\left[(Q\odot Q)^{-1}(\hep \odot \hep)\right] S^\top w_p}{w_p^\top S \Sigma S^\top w_p} \to_P 1.$$

To show the second conclusion, it suffices to prove that for $i\in[n]$,
\begin{equation}\label{hVifourmbd}\E |\hV_i - V_i |^ 4 = o(n^{-1}[\lambda_{\min}(X^\top X)]^{-4}).\end{equation} 
This together with $V_i = e_i^\top S \Sigma S^\top e_i \asymp [\lambda_{\min}(X^\top X)]^{-1}$ 
implies that for any $\ep_0>0$,
$$\P \left(\max_{i\in[n]}|\hV_i/V_i-1|>\epsilon_0 \right) \le \sum_{i=1}^n \P \left(|\hV_i/V_i-1|>\epsilon_0 \right) \le \sum_{i=1}^n \E  |\hV_i - V_i |^ 4 (\ep_0 V_i)^{-4}=o(1).$$

By \eqref{quadformZ} and using Lemma \ref{a4}, we have 
$\E |\hV_i - V_i |^ 4 \le C \left[ \E |Z_i|^4 \|G(e_i)\|_{\Fr}^4 + \E |Z_i|^8 \tr G(e_i)^4\right]$ for some positive constant $C$. Then we bound $\|G(w_p)\|_{\Fr}^4$ and $\tr G(w_p)^4$ for any $w_p$ of unit norm.
Recalling the lower bound in \eqref{lbdwShad}, we have the  upper bound for $\|(w_p^\top S) \odot  (w_p^\top S)\|$ given as
\begin{equation*}\begin{aligned}
    \|(w_p^\top S) \odot  (w_p^\top S)\| & \le \max_{j\in[n]}|w_p^\top S_{.j}|  \left(\sum_{j=1}^n  (w_p^\top S_{.j})^2 \right)^{1/2}\\
    & \le  \max_{j\in[n]}|w_p^\top S_{.j}|[\lambda_{\min}(X^\top X)]^{-1/2} = o(n^{-1/4}[\lambda_{\min}(X^\top X)]^{-1}),
\end{aligned}\end{equation*}
where the last step uses condition \ref{cond2lemWnorm} of Proposition \ref{lemWnorm}. Then using \eqref{Wifr} we can check that $$\|G(w_p)\|_{\Fr} = o(n^{-1/4}[\lambda_{\min}(X^\top X)]^{-1}).$$
Using $\|G(w_p)\| = o(n^{-1/2}[\lambda_{\min}(X^\top X)]^{-1})$---concluded from Proposition \ref{lemWnorm}---we have for any $w_p$ of unit norm
that $\tr G(w_p)^4 = o(n^{-2}[\lambda_{\min}(X^\top X)]^{-4})$. Combining these bounds with $\E|Z_i|^8 = O(n)$ for $i\in[n]$, we conclude \eqref{hVifourmbd}.
\qed 

\medskip

\subsubsection{Proof of Theorem \ref{thmbetanormal}}
By the first conclusion in Lemma \ref{lem:cons-est} and  Slustky's theorem, it suffices to show that \begin{equation}\label{cltoracle}
    \frac{w_p^\top \hat{\beta} - w_p^\top \beta}{\sqrt{w_p^\top S \Sigma S^\top w_p}} \Rightarrow \N(0,1).
\end{equation}
Let $f:\R^n\to\R$ be defined for all $z\in\R^n$ by
$f(z)= w_p^\top (\hat{\beta}-\beta) = w_p^\top (X^\top X)^{-1}X^\top \Sigma^{1/2}z$.
We have \begin{equation*}  
\E \sum_{j=1}^n \left|\frac{\partial f}{\partial z_j}(Z)\right|^4= 
\sum_{j=1}^n (w_p^\top (X^\top X)^{-1}X^\top \Sigma^{1/2} e_j)^4. \end{equation*}
Therefore, applying  Theorem 2.2 of \cite{chatterjee2009fluctuations} with $\kappa_2 =0$ and $\sigma^2 =  w_p^\top S \Sigma S^\top w_p$, along with the assumption \eqref{assdelo}, we conclude \eqref{cltoracle}.

\section{Additional simulation results}\label{sec:appendix-simulations}
This section includes additional simulation results mentioned in Section \ref{sec:num}. 

\subsection{Case 1}
Figure \ref{fig:case1-allcoordinae} displays the mean type-I error for each coordinate over 1000 simulations in Case 1.
It exhibits a similar pattern to that shown in Figure \ref{fig:allcoordinae}. 
We also plot the mean type-I error in the first and second coordinates over 1000 simulations in Figure \ref{fig:case1-bootstrap-jackknife-first}. 
Compared with case 2 shown in Figure \ref{fig:case2-bootstrap}, where MW has an inflated type-I error for the first coordinate, 
here the MW estimator is more accurate, 
though still performing slightly worse than the Hadamard-t estimator. We also observe that the Hadamard-t estimator is more accurate compared with the MW estimator for larger $p$, as also reflected by MAD reported in Figure \ref{fig:case1-allcoordinae}.

\begin{figure}
    \begin{subfigure}{.5\textwidth}
  \centering
  \includegraphics[scale=0.5]{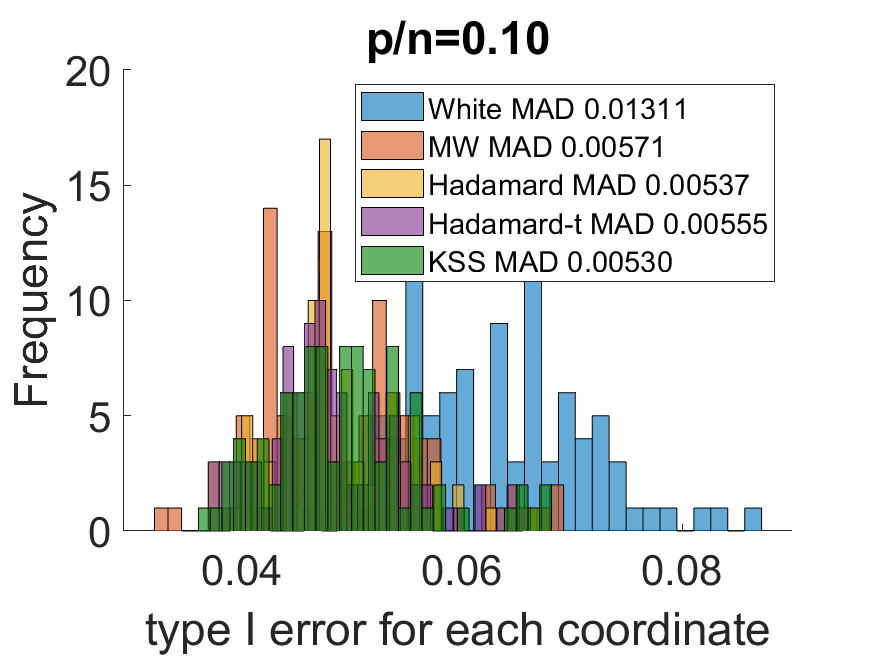}
\end{subfigure}
\begin{subfigure}{0.5\textwidth}
  \centering
  \includegraphics[scale=0.5]{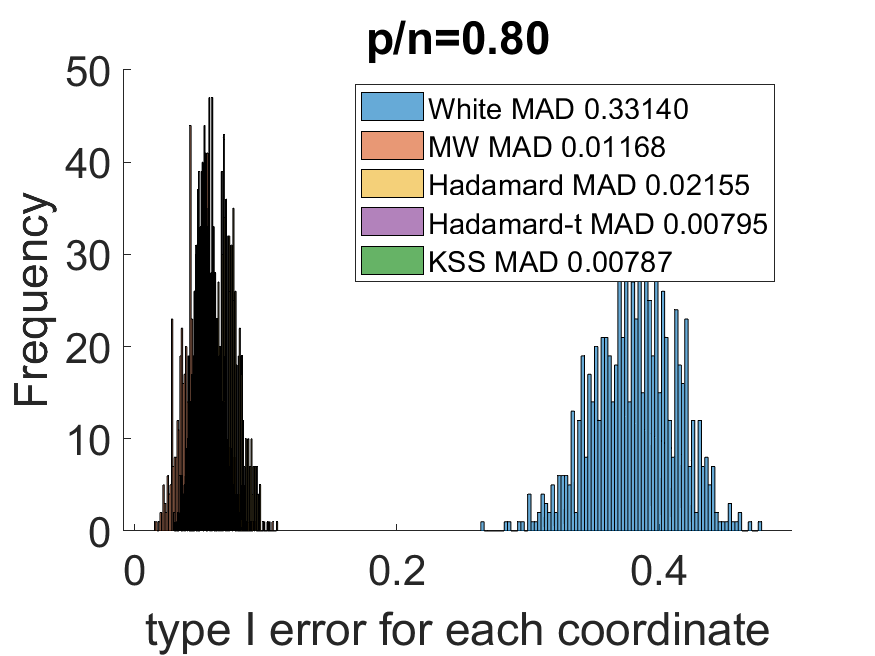}
\end{subfigure}
\caption{Mean type-I error for each coordinate over 1000 simulations for Case 1.}
    \label{fig:case1-allcoordinae}
\end{figure}

\begin{figure}
\begin{subfigure}{.5\textwidth}
  \centering
  \includegraphics[scale=0.5]{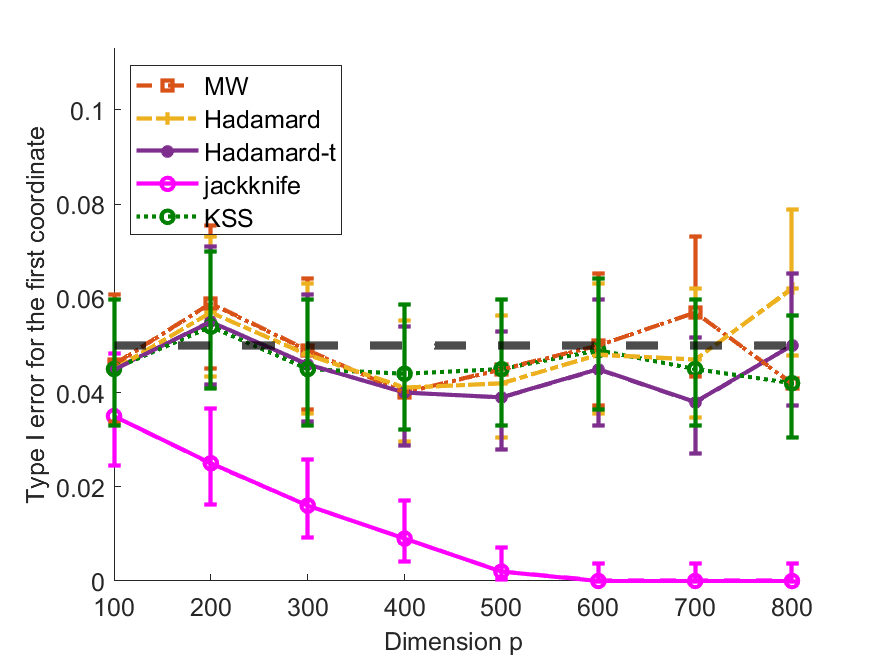}
\end{subfigure}
\begin{subfigure}{0.5\textwidth}
  \centering
  \includegraphics[scale=0.5]{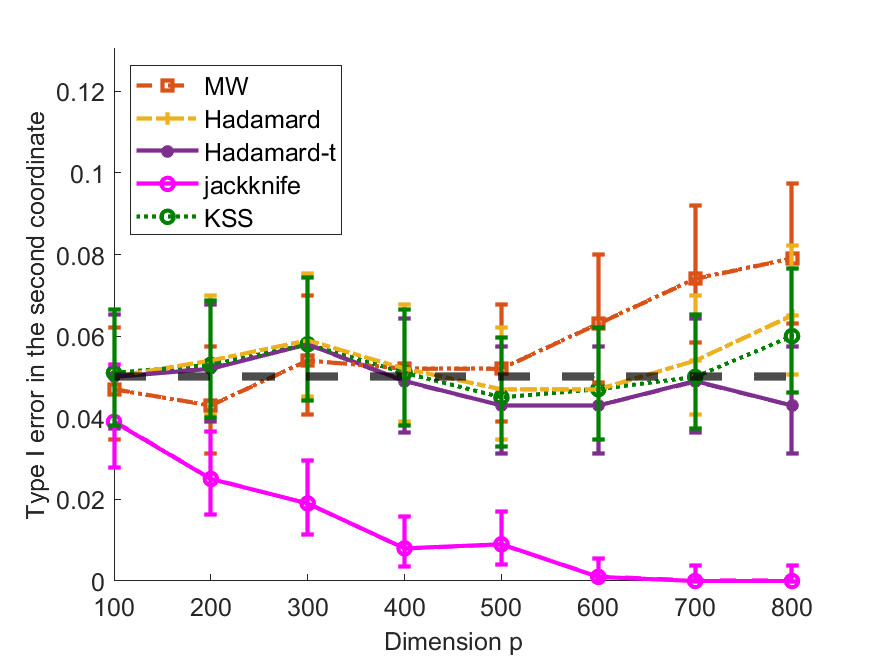}
\end{subfigure}
\caption{Mean type-I error in the first and second coordinate over 1000 simulations each for Case 1. The error bars represent 95\% Clopper-Pearson intervals for the coverage.}
    \label{fig:case1-bootstrap-jackknife-first}
\end{figure}

\subsection{Case 2}
It is observed from Table \ref{tab:cover-error-jackknife} and Figure \ref{fig:case2-bootstrap-jackknife-first} that the jackknife estimator performs poorly. 
The performance of KSS is similar to that of the Hadamard-t estimator.

\begin{table}[ht!]
\caption {Type-1 error in the first coordinate for various methods.} \label{tab:cover-error-jackknife} 
\centering
\setlength\extrarowheight{-1pt}
\begin{tabular}{ccccccccc} 
\hline
Method \textbackslash\, Dimension &100& 200& 300& 400& 500& 600& 700& 800\\
\hline
MW &0.072 &0.080 & 0.088 & 0.100 & 0.081 & 0.100 & 0.082 & 0.071\\ 
\hline
Hadamard &0.067 & 0.070 & 0.064 & 0.060 & 0.051 & 0.059 & 0.0560 & 0.056\\ 
\hline
Hadamard-t &0.063 & 0.067 & 0.061 & 0.058 & 0.045 & 0.056 & 0.048 & 0.040\\ 
\hline
jackknife &0.061 & 0.053 & 0.031 & 0.030 & 0.011 & 0.009 & 0.003 & 0.000\\
\hline
\end{tabular}
\end{table}

\vspace{3em}

\begin{figure}
\begin{subfigure}{.5\textwidth}
  \centering
  \includegraphics[scale=0.5]{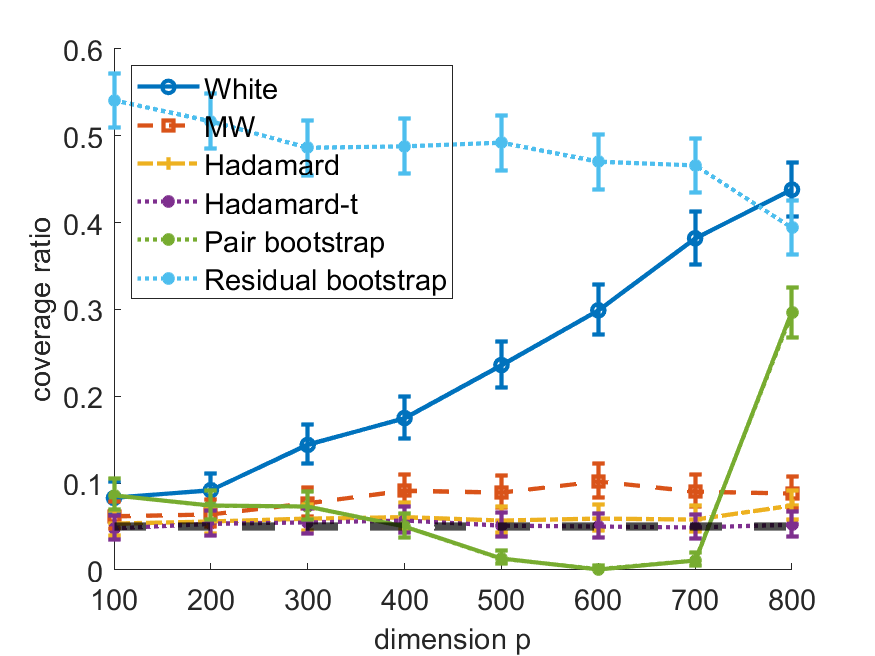}
\end{subfigure}
\begin{subfigure}{0.5\textwidth}
  \centering
  \includegraphics[scale=0.5]{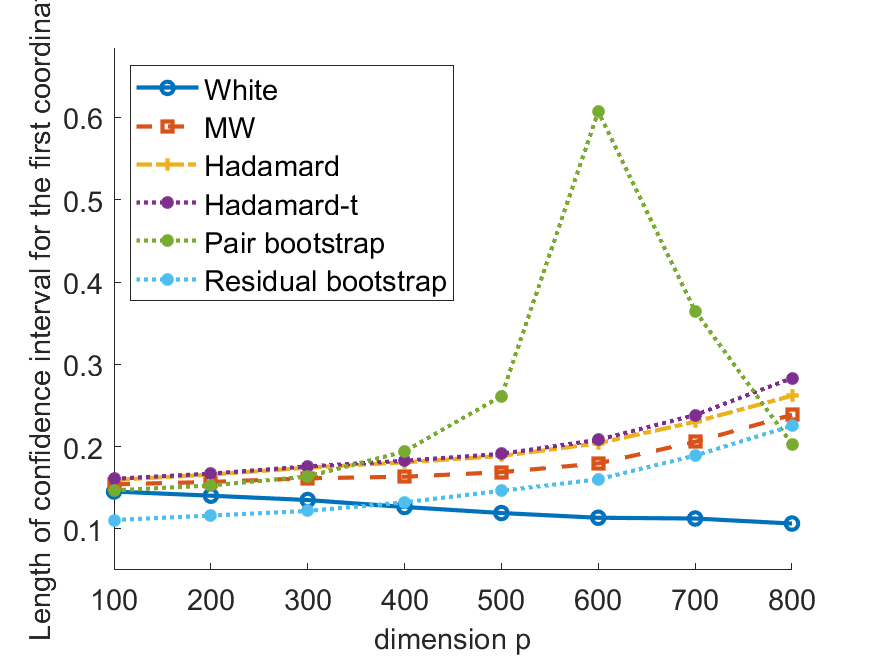}
\end{subfigure}
\caption{Results for data generated from model in Case 2. Left: Mean type-I error in the first coordinate over 1000 simulations; Right: Mean length of the confidence intervals.  The error bars represent 95\% Clopper-Pearson intervals for the coverage.}
    \label{fig:case2-bootstrap-6methods}
\end{figure}

\begin{figure}
\begin{subfigure}{.5\textwidth}
  \centering
  \includegraphics[scale=0.6]{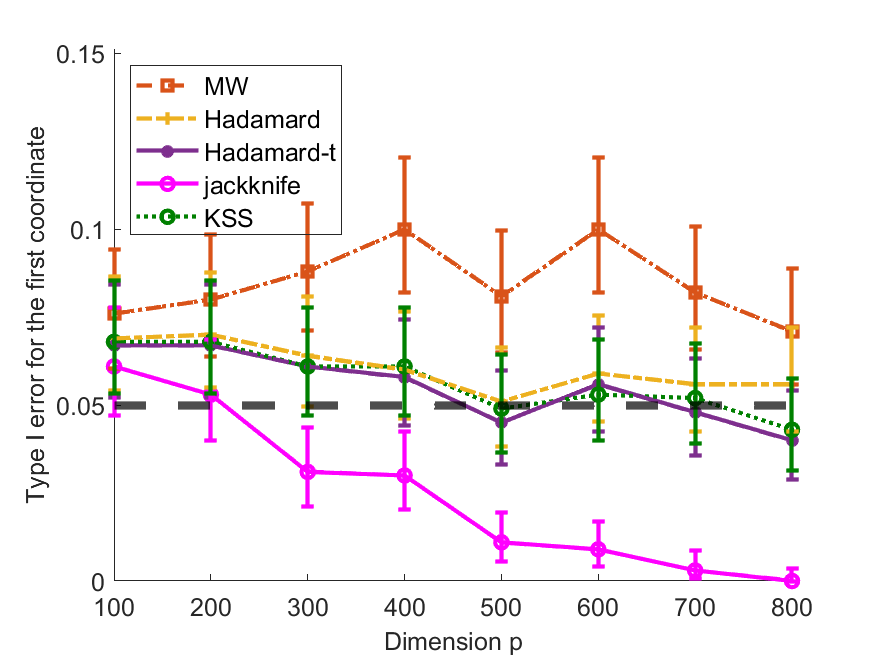}
\end{subfigure}
\begin{subfigure}{0.5\textwidth}
  \centering
  \includegraphics[scale=0.6]{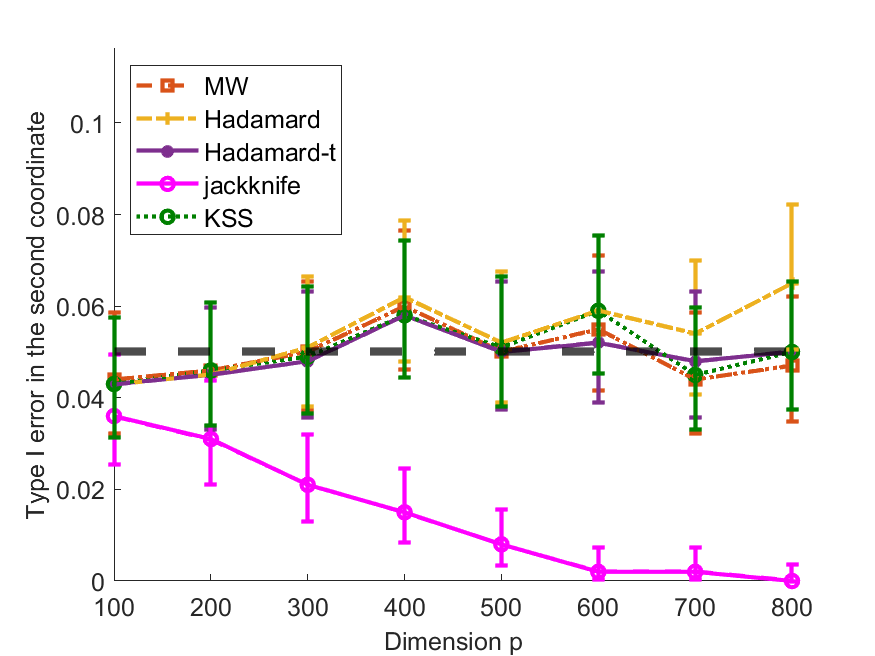}
\end{subfigure}
\caption{Mean type-I error in the first and second coordinate over 1000 simulations each for Case 2. The error bars represent 95\% Clopper-Pearson intervals for the coverage.}
    \label{fig:case2-bootstrap-jackknife-first}
\end{figure}

\vspace{3em}
Figure \ref{fig:case2-mse} illustrates the bias in estimating the MSE of the OLS estimators in Case 2. We observe the same pattern as in Case 1,
where the MW and Hadamard estimator have comparable performance and both are much better than the White estimator.
\begin{figure}
\begin{subfigure}{.5\textwidth}
  \centering
  \includegraphics[scale=0.35]{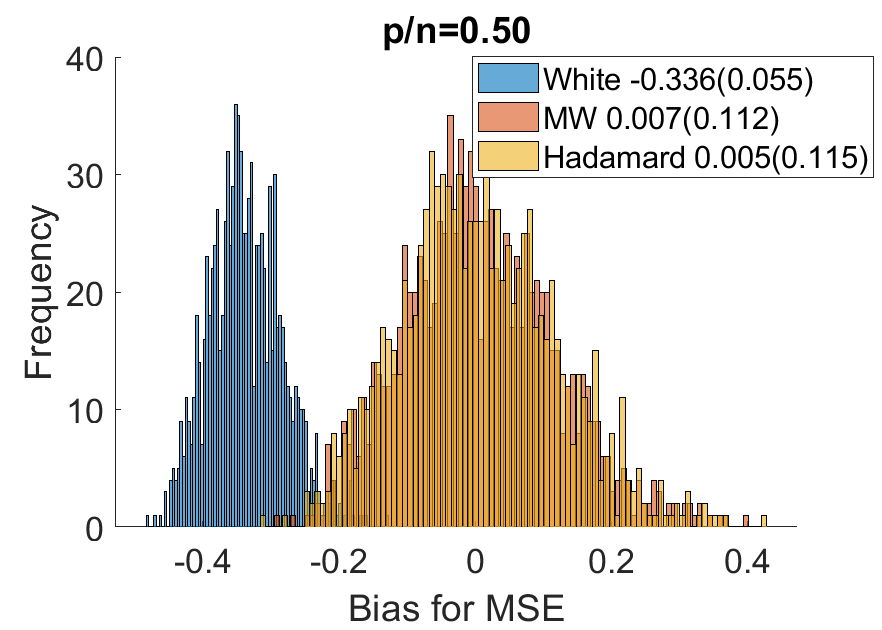}
  \caption{$p/n=0.5$}
\end{subfigure}
\begin{subfigure}{.5\textwidth}
  \centering
  \includegraphics[scale=0.35]{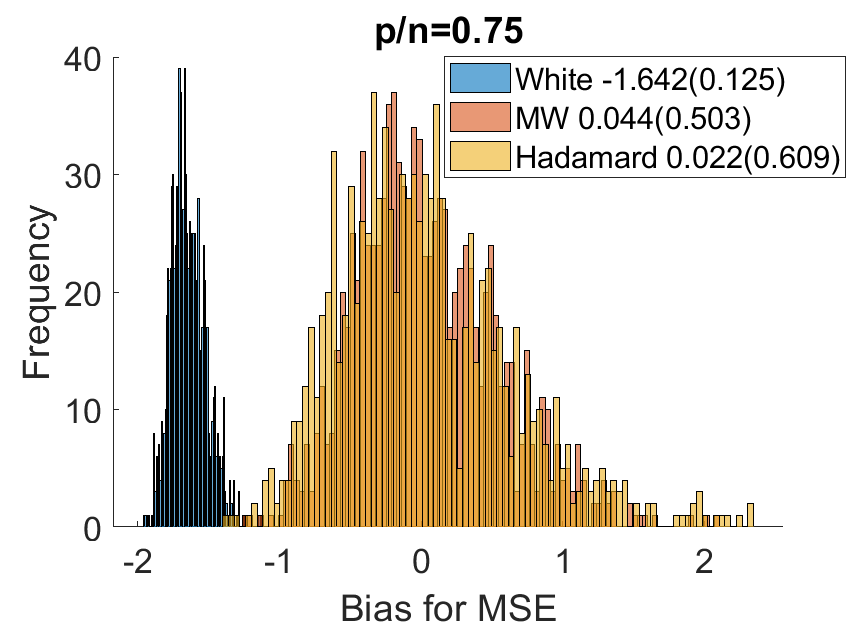}
  \caption{$p/n=0.75$}
\end{subfigure}
\caption{Bias in estimating MSE for data from Case 2.}
\label{fig:case2-mse}
\end{figure}

\end{document}